\documentclass[oneside,english]{amsart}
\RequirePackage{amsthm,amsmath,mathtools}

\usepackage[T1]{fontenc}
\usepackage{geometry}
\usepackage{subcaption}
\usepackage{enumitem}
\usepackage{babel,verbatim}
\usepackage{float}
\usepackage{amstext}
\usepackage{amssymb}
\usepackage{amsmath}
\usepackage{graphicx}
\usepackage[numbers]{natbib}
\usepackage{bbm}
\RequirePackage[colorlinks,citecolor=blue,urlcolor=blue]{hyperref}

\providecommand{\bysame}{\leavevmode\hbox to3em{\hrulefill}\thinspace}
\providecommand{\MR}{\relax\ifhmode\unskip\space\fi MR }

\providecommand{\href}[2]{#2}

\numberwithin{equation}{section}
\numberwithin{figure}{section}
\theoremstyle{plain}

\theoremstyle{plain}
\newtheorem{thm}{\protect\theoremname}
\theoremstyle{remark}
\newtheorem*{rem*}{\protect\remarkname}
\theoremstyle{remark}
\newtheorem{rem}[]{\protect\remarkname}
\theoremstyle{plain}
\newtheorem*{assumption*}{\protect\assumptionname}
\theoremstyle{plain}
\newtheorem{lem}[thm]{\protect\lemmaname}
\theoremstyle{plain}
\newtheorem{cor}[thm]{\protect\corollaryname}
\theoremstyle{plain}
\newtheorem{prop}[thm]{\protect\propositionname}
\theoremstyle{definition}
\newtheorem*{example*}{\protect\examplename}

\PassOptionsToPackage{subsection=false,section=false}{beamerouterthememiniframes}
\usepackage{tikz}
\usetikzlibrary{shapes.multipart,shapes,arrows,decorations.pathreplacing}

\newcommand*\circled[1]{\tikz[baseline=(char.base)]{%
		\node[shape=circle,fill=white!20,draw,inner sep=2pt] (char) {#1};}}
\newcommand{\1}{\mathbbm{1}}
\renewcommand{\O}{\mathcal{O}}

\newcommand{\D}{\mathcal{D}}
\newcommand{\F}{\mathcal{F}}
\newcommand{\E}{\mathbb{E}}

\newcommand{\p}{\mathbb{P}}

\newcommand{\s}{\mathbb{S}}
\newcommand{\R}{\mathbb{R}}
\newcommand{\N}{\mathbb{N}}
\newcommand{\U}{{\mathrm U}}
\newcommand{\ov}[1]{\overline{#1}}
\newcommand{\eqd}{{\overset{d}{=}}}
\newcommand{\ve}{\varepsilon}

\newcommand{\FV}[1]{C_b{#1}}

\newcommand{\indep}{\scalebox{.9}[1.25]{$\perp \!\!\! \perp$}} 

\usepackage{lineno}
\usepackage{pgfplots} 

\newcommand{\BrownianMotion}[7]{
	\draw[#6] (#1,#2)
	\foreach \x in {1,...,#3}
	{   -- ++(#4,rand*#5)
	}
	node[below right] {#7};
}

\usepackage{caption}
\makeatother
\providecommand{\algorithmname}{Algorithm}
\providecommand{\assumptionname}{Assumption}
\providecommand{\examplename}{Example}
\providecommand{\lemmaname}{Lemma}
\providecommand{\propositionname}{Proposition}
\providecommand{\remarkname}{Remark}
\providecommand{\corollaryname}{Corollary}
\providecommand{\theoremname}{Theorem}

\begin{document}
	
	\title[Joint density of a stable process and its supremum]{Joint density of the stable process and its supremum: regularity and upper bounds}
	
	\author{Jorge Gonz\'{a}lez C\'{a}zares, Arturo Kohatsu Higa \and Aleksandar Mijatovi\'{c}}
	
	\address{Department of Mathematics, University of Warwick, \& The Alan Turing Institute, UK}
	
	\email{jorge.gonzalez-cazares@warwick.ac.uk}
	
	\address{Ritsumeikan University, Nojihigashi, Kusatsu, Shiga, Japan}
	
	\email{khts00@fc.ritsumei.ac.jp}
	
	\address{Department of Statistics, University of Warwick, \& The Alan Turing Institute, UK}
	
	\email{a.mijatovic@warwick.ac.uk}
	
	\begin{abstract}
		This article uses a combination of three ideas from simulation to establish a nearly optimal polynomial upper bound for the joint density of the stable process and its associated supremum at a fixed time on the entire support of the joint law. The representation of the concave majorant of the stable process and the Chambers-Mallows-Stuck representation for stable laws are used to define an approximation of the random vector of interest.  An interpolation technique using multilevel Monte Carlo is applied to 
	 accelerate the approximation, allowing us to establish the infinite differentiability of the joint density as well as nearly optimal polynomial upper bounds for the joint mixed derivatives of any order.  
	\end{abstract}
	
	\maketitle
	
	\section{Introduction}
	\label{sec:intro}
	Let $(X_t)_{t\ge0}$ be a non-monotonic 
	$\alpha$-stable process with 
	$\alpha\in(0,2)$ and positivity parameter 
	\begin{align}
		\label{cond:rho}
		\rho:=\p(X_1>0)\in[1-1/\alpha,1/\alpha]\cap(0,1). 
	\end{align}
	For any fixed $T>0$, denote by $\ov{X}_T:=\sup_{s\in [0,T]}X_s$ its 
	supremum over the time interval $[0,T]$. Our main result, 
	Theorem~\ref{thm:density_bound} below, provides the regularity and upper 
	bounds for the joint density of $(X_T,\ov X_T)$ and its derivatives of any 
	order. These explicit polynomial bounds, valid on the entire support set of the 
	joint law, are nearly optimal. For a detailed explanation, see the discussion following Theorem~\ref{thm:density_bound}. 
	
	The joint law of 
	$(X_T,\ov X_T)$
	arises in the scaling limit of many stochastic models, including queues with 
	heavy-tailed workloads 
	(see~\cite[\S 5.2]{MR3379923} and the references therein). 
	In such cases,
	the bounds in Theorem~\ref{thm:density_bound} are necessary for the construction of 
	the asymptotic confidence intervals.
	Moreover, in some prediction problems (e.g.~\cite{bernyk2011}), regularity of the density of 
	$\ov{X}_T-X_T$, established in Theorem~\ref{thm:density_bound}, is important. 
	
	 Our approach is rooted in recent advances in simulation used to build an efficient approximation of the law of $(X_T,\ov X_T)$. More precisely, we use the representation of the concave majorant of a stable processes, recently applied in~\cite{LevySupSim} to construct a geometrically convergent simulation algorithm for sampling from the law of $(X_T,\ov X_T)$. In order to analyze the regularity of this joint law, we express the stable random variables arising in the concave majorant representation of the supremum $\ov X_T$ via the classical Chambers-Mallows-Stuck representation. To the best of our knowledge, this approach to study the regularity and upper bounds of the densities of the joint law differs from the probabilistic and analytical techniques applied in this context in the literature so far (see \href{https://youtu.be/x0n3Up9CxCA}{YouTube}~\cite{Presentation_AM} for a short presentation of our results and techniques).
	
	In general, it is well-known that the properties of the approximation do not necessarily persist in the limit (see \cite{bhat} for a comprehensive study in the case of the central limit theorem). 
	In our case, in order to establish regularity and achieve nearly optimal upper bounds of the limit law, we accelerate the convergence of the approximation procedure using ideas behind the multilevel Monte Carlo method. This method has been successfully applied in Monte Carlo estimation (see \cite{MR2436856} and the references in the  \href{https://people.maths.ox.ac.uk/gilesm/mlmc_community.html}{webpage}) to reduce the computational complexity of the algorithm for a pre-specified level of accuracy. In theoretical terms, we apply 
	 the multilevel idea as an interpolation methodology (we have not been able to find multilevel Monte Carlo methods used for this purpose in the literature). Other interpolation techniques applied to stochastic equations are found in \cite{bally}, see also the references therein. 
	In fact, the authors in~\cite{bally} use a different interpolation technique to obtain qualitative properties using approximation methods. In the examples they treat, it is hard to tell if they achieve optimal results. In our case, the near optimality is due to the geometrical convergence of the approximation of the joint law based on the concave majorant, see~\cite{LevySupSim}.
	
	

	To the best of our knowledge, only the regularity of the density of the 
	marginals of $(X_T,\ov X_T)$ has been considered so far.
	The first component $X_T$ follows a stable law, which is very well understood 
	(see e.g.~\cite{Sato} and the references therein). 
	In fact, it is known that the density of $X_T$
	has the following asymptotic behavior 
	\begin{align}
		\mathbb{P}({X}_T\in dx)\stackrel{x\to\infty}{\approx }
		Tx^{-\alpha-1}dx;\quad 
		\mathbb{P}({X}_T\in dx)\stackrel{x\to 0}{\approx } T^{-1/\alpha}dx.
		\label{eq:ax}
	\end{align}
	Even though 
	its law has been the focus of a number of papers over the past seven decades
	(starting with Darling~\cite{MR80393}, Heyde~\cite{MR251766} and 
	Bingham~\cite{MR0415780}),
		far less information is available about the density of the second component, $\ov X_T$,
	which is a functional of the path of a stable process.  
	 Most of the results about the law 
	of the supremum of a L\'evy process rely on the Wiener-Hopf 
	factorization and/or the equivalence with laws related to excursions of 
	reflected processes~\cite{MR3531705,MR4216728}. For example, in~\cite{MR3098676}, the author obtains 
	explicit formulae for the supremum  in the spectrally negative stable and 
	symmetric Cauchy cases. The smoothness of the density of the supremum 
	$\ov X_T$ is known, see e.g.~\cite[Thm~2.4~\&~Rem.~2.14]{MR3835481}.
	
	The papers~\cite{MR2402160,MR2599201} study the asymptotic behaviour 
	of the density of the supremum at infinity and at zero. In~\cite{MR2599201}, 
	the authors rely on  local times and excursion theory, the Wiener-Hopf 
	factorisation and a distributional connection between stable suprema and 
	stable meanders. 
	Power series expansions of the density of $\ov X_T$ have 
	been established in~\cite{MR2789582,MR3005012} in some particular situations. 
	Since stable processes are self-similar and Markov, results 
	in~\cite{MR3835481} can be used to deduce the asymptotic behaviour of 
	the density (and its derivatives) of $\ov{X}_T$, see the paragraph following 
	Corollary~\ref{cor:sup} below. 
	
	In short, the proofs of the results obtained so far in the literature rely on excursion theory or 
	the Wiener-Hopf factorisation. These methods
	exploit the independence of 
	$\ov{X}_{\mathrm{e}}$ and $\ov{X}_{\mathrm{e}}-X_{\mathrm{e}}$ over 
	an independent exponential time horizon $\mathrm{e}$. 
	The dependence of all of the above methods 
	on a number of specific analytical identities for the law of $\ov{X}_T$ makes 
	them hard to generalise to the law of $(X_T,\ov{X}_T)$.
	
	The result closer to our study are the asymptotics established in~\cite{MR2599201,MR3835481}: 
	\begin{align}
		\mathbb{P}(\ov{X}_T\in dy)\stackrel{y\to\infty}{\approx }
		Ty^{-\alpha-1}dy; \quad
		\mathbb{P}(\ov{X}_T\in dy)\stackrel{y\to 0}{\approx}
		T^{-\rho}y^{\alpha\rho-1}dy.  
		\label{eq:as}
	\end{align}
	Taking into consideration the asymptotics in \eqref{eq:ax}--\eqref{eq:as}, it is natural that the asymptotics for the law of $(X_T,\ov{X}_T)$ are determined by four sub-domains in the support  $\O:=\{(x,y)\in\R^2: y> \max\{x,0\}\}$.
	Our upper bound on the joint density and its derivatives, illustrated in Figure \ref{fig:regions} below, is close to optimal in the sense that we obtain such a result for any $\alpha'$ arbitrarily close to $\alpha$ featuring in\eqref{eq:ax}--\eqref{eq:as}. The reason why we are unable to obtain the result for the choice $\alpha'=\alpha$ is technical and due to the use of moments to bound tail behaviours in the spirit of Markov's and Chebyshev's inequalities.

	Malliavin calculus is a long 
	developed subject in the area of stochastic analysis of jump processes. 
	The ultimate goal of the general theory is to obtain an infinite dimensional 
	calculus with the view of investigating random quantities generated by the 
	jump process and, in particular, the regularity of the law of path 
	functionals of the process (see e.g.~\cite{bich,nualart2} for a general 
	reference). Notably, these theoretical developments in Malliavin 
	calculus have fallen short of the problem of the regularity of the density of 
	$\ov X_T$, because the supremum of a jump process (as a random variable) 
	appears not to depend smoothly on the underlying jumps. 
	An exception is the result in~\cite{Bouleau}, where the authors rely on the 
	Lipschitz property of the supremum functional to prove the existence 
	of a density for the supremum of a jump process in a general class, using the 
	so-called lent-particle method. However, since $\ov X_T$ is not a smooth 
	functional of the path, it is unclear how to apply these methods to analyse 
	the regularity and behavior of the density near the boundary of its support.

	The approach used in this article does not fall in any of the above categories of Malliavin Calculus, nor does it rely on any results from Malliavin Calculus of jump processes. More precisely, we do not use infinite dimensional objects but only study limits of finite collections of random variables, arising in the noise used in our representation of the law of $(X_T,\ov{X}_T)$. Our main underlying idea is to exploit the geometrically convergent approximation of the random vector of interest, 
	establish the required properties of the densities for the approximate 
	vectors and prove that these properties persist in the limit. In this sense, our approach is both self-contained and elementary.
	
	More specifically, we establish a probabilistic 
	representation for the joint density of $(X_T,\ov{X}_T)$ and its derivatives in  Theorem~\ref{thm:series} below,
	based on a telescoping sum of successive approximations analogous to the 
	multilevel method (cf.~\cite{MR2436856}). 
	The telescoping sum formula for the density and its derivatives is based on 
	an elementary integration-by-parts formula for successive finite dimensional approximations of 
	$(X_T,\ov{X}_T)$. These approximations are \textit{not} using the 
	path of the stable process $(X_t)_{t\in[0,T]}$ directly as would be the case in Malliaving Calculus for processes with jumps. Instead, the concave majorant of 
	$(X_t)_{t\in[0,T]}$, given in~\cite[Thm~1]{MR2978134}, is used to represent 
	$(X_T,\ov{X}_T)$ as an infinite series~\cite{LevySupSim,MR4032169}.
	The terms in this series are the increments of the stable process over 
	macroscopic (but geometrically small) time steps given by an independent 
	stick-breaking process on $[0,T]$ (for more details, see Section \ref{sec:3.1}). 
	We then build our elementary finite-dimensional integration-by-parts formulae for the partial sum 
	approximations of $(X_T,\ov{X}_T)$ using the scaling property of stable 
	increments and their Chambers-Mallows-Stuck representation~\cite{Weron}, which 
	in the non-Cauchy case $\alpha\ne 1$, amounts to a semi-linear function of independent 
	uniform and exponential variables,  Section~\ref{sec:3.1}.

	
	
	\subsection{Organisation}
	
	The remainder of the paper is organised as follows. In Section~\ref{sec:2} 
	we present Theorem~\ref{thm:density_bound}, the main result of the paper, 
	and some applications of these results. 
	Subsection~\ref{sec:3.1} introduces the technical
	notation for the proofs and Subsection~\ref{subsec:Ibpf} establishes the Ibpf. At the end of this section, we also give an important technical Proposition \ref{prop:Theta_bound} which gives all the bounds needed in order to be applied in the Ibpf formula obtained. 
	In Section \ref{sec:4}, we give the proof of our main result, Theorem~\ref{thm:density_bound}, using the ingredients developed in previous sections. This proof uses the interpolation method in the sense that the approximation method based on the convex majorant converges geometrically fast while the density bounds explode polynomially.  Combining these two characteristics one obtains the almost optimal bounds. 
	
	We close the article with some technical appendices which prove the important technical Proposition \ref{prop:Theta_bound}. The proof of this proposition is composed of algebraic inequalities  which are obtained in Subsection \ref{subsec:alg}. The upper bounds are products of powers of basic random variables. After the proof we give also a heuristic interpretation of a basic interpolation technique  used in the estimation of the moments.
 Finally, the moment estimates are obtained in Section \ref{sec:mom}. Throughout the article we concentrate on the case $\alpha\neq 1$ leaving the special Cauchy case, $\alpha=1$ to the Appendix \ref{app:cauchy}. 
	
	%
	%
	%
	
	Section~\ref{sec:conclusions} 
	concludes the paper, remarking on our techniques and methodology as well 
	as possible extensions. Appendices collects relevant bounds
	on the moments of a stick-breaking process and the moment generating function for powers of exponentially distributed random variables.
	
	\section{Main result and applications}
	\label{sec:2} 
	
	
	As explained in the Introduction, we give first our main result:
	\begin{thm}
		\label{thm:density_bound}
		Assume that $\alpha\in(0,2)$.
		Let $F(x,y):=\p(X_T\le x,\ov X_T\le y)$ be the distribution function of 
		$(X_T,\ov X_T)$. The joint density of $F$ exists and is infinitely differentiable 
		on the open set $\O$. Moreover, for any fixed 
		$n,m\ge 1$ and $\alpha'\in[0,\alpha)$ there is some $C>0$ such that for all 
		$x,y>0$ and $T>0$, we have 
		\begin{equation}
			\label{eq:density_at_infinity}
			\begin{split}
				|\partial_x^{n}\partial_y^{m}F(x,y)|
				&\le Cy^{-m}(y-x)^{1-n-m}(2y-x)^{m-1}\\
				&\qquad\times\min\big\{f^{00}_{\alpha'}(x,y),f^{01}_{\alpha'}(x,y),
				f^{10}_{\alpha'}(x,y),f^{11}_{\alpha'}(x,y)\big\},
			\end{split}
		\end{equation}
		where 
		$f^{ij}_{\alpha'}(x,y)
		:=T^{\frac{\alpha'}{\alpha}(i(2-\rho)+j(1+\rho)-1)}
		(y-x)^{\alpha'(1-\rho)-i\alpha'(2-\rho)}
		y^{\alpha'\rho-j\alpha'(1+\rho)}$
		for
		$i,j\in\{0,1\}$.
	\end{thm}
	
	Theorem~\ref{thm:density_bound} presents a bound on the mixed derivatives 
	of the joint density of $(X_T,\ov X_T)$. The decay of the bound as $y$ 
	tends to either infinity or zero is almost sharp in the following sense: if one 
	sets $n=1$ and $\alpha'=\alpha$ in~\eqref{eq:density_at_infinity} 
	(cf. Figure~\ref{fig:regions} below)
	and integrates out 
	$x$ over $\R$, the decay of the obtained bound 
	matches the actual asymptotic behaviour of the density of $\ov{X}_T$ known 
	from the literature~\cite{MR2599201,MR2789582,MR3005012}. That is, marginals of the above bounds match the estimates in \eqref{eq:ax} and  \eqref{eq:as}. In fact, the bound in 
	Corollary~\ref{cor:sup} below is established in this way. The constant 
	$C$ in~\eqref{eq:density_at_infinity} can be made explicit. Instead of giving 
	a formula for $C$, which would be lengthy and suboptimal 
	(cf. Remark~\ref{rem:Theta_bound}(i) below), we point out that $(\alpha-\alpha')C$ 
	remains bounded as $\alpha'\uparrow\alpha$. An alternative way to understand the  optimality property is through a change of variables in equation  
	\eqref{eq:density_refl_proc} which will be proven in Section \ref{sec:4}.
	
	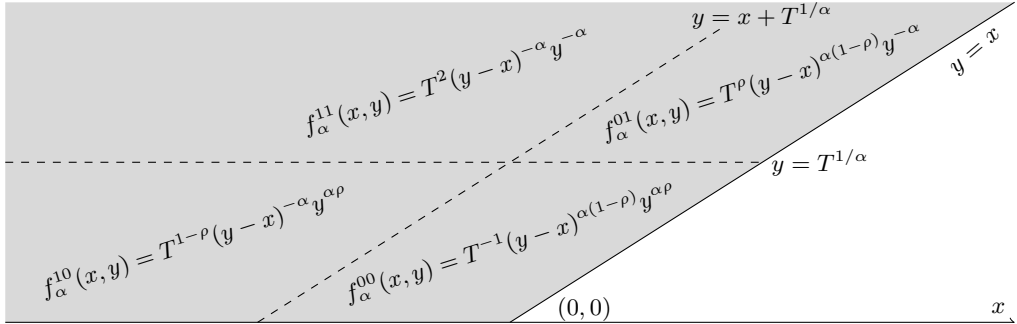
\begin{figure}[H]
		\begin{tikzpicture} 
			\begin{axis} 
				[xmin=-4,
				xmax=4,
				ymin=0,
				ymax=4.15,
				xlabel=$x$, 
				ylabel=$y$, 
				width=15cm,
				height=6cm,
				axis on top=true,
				ytick=\empty,
				xtick=\empty,
				xticklabels=none,
				yticklabels=none,
				axis x line=none, 
				axis y line=none, 
				legend pos=north east ] 
				
				\draw [draw=none,fill=gray!30] plot [smooth, samples=100, domain=0:4] (axis cs: \x, \x) -| (0,0) -- cycle;
				\draw [-, black, solid] (axis cs: 0,0) -- (axis cs: 4,4);
				\node [below] at (axis cs: 3.7,3.8) {\rotatebox{40}{\small $y=x$}};
				\draw [->, black, solid] (axis cs: -4,0) -- (axis cs: 4,0);
				\node [above left] at (axis cs: 4,0) {\small $x$};
				
				\draw [black, dashed] (axis cs: -2,0) -- (axis cs: 1.7,3.7);
				\draw [black, dashed] (axis cs: -4,2) -- (axis cs: 2,2);
				
				\node at (axis cs: 0,1) {\rotatebox{17.5}{\small$f^{00}_{\alpha}(x,y)
						=T^{-1}(y-x)^{\alpha(1-\rho)}y^{\alpha\rho}$}};
				\node at (axis cs: 2,3) {\rotatebox{17.5}{\small$f^{01}_{\alpha}(x,y)
						=T^{\rho}(y-x)^{\alpha(1-\rho)}y^{-\alpha}$}};
				\node at (axis cs: -2.5,1) {\rotatebox{17.5}{\small$f^{10}_{\alpha}(x,y)
						=T^{1-\rho}(y-x)^{-\alpha}y^{\alpha\rho}$}};
				\node at (axis cs: -.5,3) {\rotatebox{17.5}{\small$f^{11}_{\alpha}(x,y)
						=T^{2}(y-x)^{-\alpha}y^{-\alpha}$}};
				
				\node [above right] at (axis cs: .3,-.1) {\small$(0,0)$};
				
				\node [right] at (axis cs: 2,2) {\small $y=T^{1/\alpha}$};
				\node [below] at (axis cs: 2,4.12) {\small $y=x+T^{1/\alpha}$};
			\end{axis}
		\end{tikzpicture}
		\caption{\label{fig:regions}
			\small
			The set $\O=\{(x,y)\in\R^2: y> \max\{x,0\}\}$ (shaded in the figure) 
			is the support of the joint density of $(X_T,\ov{X}_T)$. 
			According to Theorem~\ref{thm:density_bound}, the support can be 
			partitioned into 4 sub-regions according to which of the 
			functions $f^{ij}_\alpha$, $i,j\in\{0,1\}$, is the smallest in the 
			(optimal) case $\alpha'=\alpha$.} 
	\end{figure}
	
	Theorem~\ref{thm:density_bound} above suggests that the asymptotic behaviour of the joint density at $(x,y)$ of $(X_T,\ov{X}_T)$ as $T\to 0$ is proportional to $T^2(y-x)^{-\alpha}y^{-\alpha}$, see Figure~\ref{fig:regions}. This is corroborated by the results in~\cite{MR3098676,MR4216728} as we now explain. Recall from~\cite[Thm~6]{MR3098676} that the density of $(\ov X_T,\ov X_T-X_T)$ satisfies  
	\begin{equation}
	\label{eq:Loic}
	\p(\ov{X}_T\in dx,\ov{X}_T-X_T\in dy)
	=dxdy\int_0^1q^*_{sT}(x)q_{(1-s)T}(y)Tds,
	\end{equation}
	where $q^*_t$ (resp. $q_t$) is the entrance density of the excursion measure of the reflected process of $X$ (resp. $-X$). By~\cite[Thm~3.1]{MR4216728} and~\cite[Ex.~3]{MR3098676}, we deduce that, as $T\to0$, the quantities $q^*_{sT}(x)/(T^\rho s^\rho x^{-\alpha-1})$
	and 
	$q^*_{(1-s)T}(x)/(T^{1-\rho}s^{1-\rho}y^{-\alpha-1})$ have positive finite limits that depend neither on $s$ nor $(x,y)$. Thus the integral on the right-hand side of~\eqref{eq:Loic} is proportional to $T^2x^{-\alpha-1}y^{-\alpha-1}$ as predicted the bound in Theorem~\ref{thm:density_bound} (see also~\eqref{eq:density_refl_proc} below).

	Setting $n=1$ and explicitly integrating in $y$ over $(0,\infty)$ yields the 
	following bounds.
	
	\begin{cor}
		\label{cor:sup}
		Assume that $ \alpha\in(0,2)$. Then the distribution function 
		$F(y):=\p(\ov{X}_T\le y)$ is infinitely smooth on $(0,\infty)$ and, for every 
		$\alpha'\in[0,\alpha)$ and $n\ge 1$, there exists some constant $C>0$ such 
		that for all $y>0$ and $T>0$, we have 
		\[
		|\partial_y^n F(y)|
		\le C y^{-n}\min\big\{T^{\frac{\alpha'}{\alpha}}y^{-\alpha'},
		T^{-\frac{\alpha'}{\alpha}\rho}y^{\alpha'\rho}\big\}.
		\]
		Define $ \tau_{y_0}:=\inf\{t>0: X_t>y_0\} $, $ y_0>0 $. Then the distribution function of $ \tau_{y_0} $ is infinitely smooth on $ (0,\infty) $ and the following estimate is satisfied for $ n\geq 1 $:
		\begin{align*}
			|\partial_T^n\p(\tau_{y_0}\leq T)|\leq CT^{-\frac 1\alpha-n}
			\times\min\{T^{\frac {\alpha'}{\alpha}}y_0^{-\alpha'} ,1\}.
		\end{align*}
	\end{cor} 
	
	It has been pointed out to us~\cite{eSavov} that the bound in 
	Corollary~\ref{cor:sup} for $\alpha'=\alpha$ can be obtained from the 
	literature. By studying the Mellin transform of 
	$\ov{X}_T$~\cite[Thm~2.4]{MR3835481} (via a distributional identity linking 
	$\ov{X}_T$ to an exponential integral arising in the Lamperti representation 
	of self-similar Markov processes~\cite[Rem.~2.14]{MR3835481}), 
	one obtains the asymptotic behaviour in \eqref{eq:as}. Similar bounds can be obtained for the 
	derivatives of the density, implying Corollary~\ref{cor:sup}.
	
	Other consequences of our main Theorem \ref{thm:density_bound} can also be derived such as the following result which reveals a complex interplay between the final value of the stable process and its supremum in the interval $[0,T]$. 
	
	\begin{cor}
		Assume that $ \alpha\in (0,2) $ and let $ y_0\geq T^{1/\alpha} $, $ x_0\leq 0 $. Then for any $ \alpha'\in (0,\alpha) $
		\begin{align*}
			\p(X_T\leq x_0, \tau_{y_0}<T)\leq CT^{2\frac\alpha{\alpha'}}y_0^{-\alpha'}
			\times 	\min\{y_0^{-\alpha'},(-x_0)^{-\alpha'} \}
		\end{align*}
	\end{cor}
	\begin{proof}The inequalities are obtained by direct integration of the bound in Theorem \ref{thm:density_bound}. That is,
		\begin{align*}
			\p(X_T\leq x_0, \overline{X}_T>y_0)\leq 
			CT^{2\frac{\alpha}{\alpha'}}\int_{\frac{L}{-x_0}}^\infty w^{-1-\alpha'}(1+w)^{-\alpha'}dw.
		\end{align*}
		From here, the result follows.
	\end{proof}

	We conclude the section by remarking on the excluded cases: 
	our methods apply to the Brownian motion case $\alpha=2$, 
	but the result is not relevant since the density of $(X_T,\ov{X}_T)$ is 
	known explicitly; 
	in~\eqref{cond:rho} we exclude $\rho\in\{0,1\}$ as in those cases the 
	monotonicity of paths implies
	$\ov{X}_T=X_T$ (resp. $\ov{X}_T=X_0$) a.s. if 
	$\rho=1$ (resp. $\rho=0$). 
	
	\section{Tools: Approximation method and sequential Ibpf}
	\label{sec:3}
	In order to avoid cumbersome multiple case studies, we will assume $\alpha\neq 1$ from now on until the last section in the Appendix where the appropriate changes for the case $\alpha=1$ will be explained. 
	
	\subsection{Approximation method for $(X_T, \ov{X}_T)$}
	\label{sec:3.1}

	Throughout the article, we fix $T>0$ and we will use the following decomposition of the random variable $X_T$ using $X_+:=\ov{X}_T$ and $X_-$  which denote the supremum of 
	$(X_t)_{t\in[0,T]}$ and its reflected process $X_-:=X_+-X_T$. Therefore instead of working with $(X_T,\ov{X}_T)$, we will use $(X_+,X_-)$ whose law is supported in $\mathbb{R}_+^2$ This $\pm$ notation will be useful in order to write dual formulas that are valid for both random variables $X_\pm$.  
	
	In fact, the proof of Theorem~\ref{thm:density_bound} studies the equivalent pair 
	$(X_+,X_-)$, instead of $(X_T,\ov{X}_T)$, and shows the following: let 
	$\widetilde F(x,y):=\p(X_+\le x, X_-\le y)$, then for any 
	$\alpha'\in[0,\alpha)$ and $n,m\ge 1$ there exists some constant $C>0$ 
	such that for any $T,x,y>0$ we have 
	\begin{equation}
		\label{eq:density_refl_proc}
		|\partial_x^n\partial_y^m \widetilde F(x,y)|
		\le Cx^{-n}y^{-m}\min\big\{T^{\frac{\alpha'}{\alpha}}x^{-\alpha'},
		T^{-\frac{\alpha'}{\alpha}\rho}x^{\alpha'\rho}\big\}
		\min\big\{T^{\frac{\alpha'}{\alpha}}y^{-\alpha'},
		T^{-\frac{\alpha'}{\alpha}(1-\rho)}y^{\alpha'(1-\rho)}\big\}.
	\end{equation}

	For this reason, we will use in many formulas multiple $\pm$ and $\mp$ signs. It is assumed that the signs match, i.e., 
	all $\pm$ are $+$ (resp. $-$) and all $\mp$ are $-$ (resp. $+$) 
	simultaneously. For example, $A_\pm=\mp B_\mp$ if and only if 
	$A_+=-B_-$ and $A_-=+B_+$. Additionally, we use the notation 
	$[x]^+=\max\{x,0\}$ and $[x]^-=\max\{-x,0\}$. We stress that if the brackets 
	are not present, then the notation refers to a different object. For example, $X_{\pm,n}$ denote the approximations for $X_{\pm}$ respectively and $\D^\pm_n$ are the associated derivative operators to be defined below. 
	Finally, we denote $x\wedge y=\min\{x,y\}$ and $x\vee y=\max\{x,y\}$. 
	
	We will use an approximation method for the pair $(X_T,\ov{X}_T)$ used in~\cite[\S 4.1]{LevySupSim} (see also~\cite[Eq.~(2.2)]{MR4032169} 
	and~\cite[Thm~1]{MR2978134}) which is based on the concave majorant of $X$, see Figure~\ref{fig:FacesCM}.
	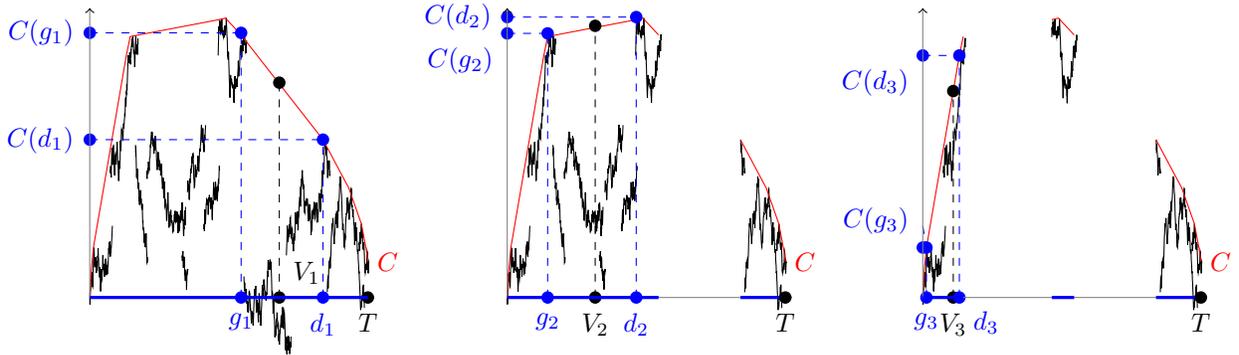
\begin{figure}[ht]
		\begin{tikzpicture}
			\pgfmathsetseed{101101}
			\BrownianMotion{0}{0}{100}{0.003}{0.2*.7}{black}{};
			\BrownianMotion{101*.0027}{2.5*.7}{123}{0.003}{0.2*.7}{black}{};
			\BrownianMotion{225*.0027}{2.9*.7}{54}{0.003}{0.2*.7}{black}{};
			\BrownianMotion{280*.0027}{2.4*.7}{172}{0.003}{0.2*.7}{black}{};
			\BrownianMotion{453*.0027}{0.8*.7}{27}{0.003}{0.2*.7}{black}{};
			\BrownianMotion{481*.0027}{2.0*.7}{80}{0.003}{0.2*.7}{black}{};
			\BrownianMotion{562*.0027}{1.3*.7}{69}{0.003}{0.2*.7}{black}{};
			\BrownianMotion{632*.0027}{4.6*.7}{126}{0.003}{0.2*.7}{black}{};
			\BrownianMotion{759*.0027}{0.4*.7}{208}{0.003}{0.2*.7}{black}{};
			\BrownianMotion{968*.0027}{0.8*.7}{67}{0.003}{0.2*.7}{black}{};
			\BrownianMotion{1036*.0027}{1.8*.7}{132}{0.003}{0.2*.7}{black}{};
			\BrownianMotion{1169*.0027}{0.1*.7}{164}{0.003}{0.2*.7}{black}{};
			\BrownianMotion{1324*.0027}{1.0*.7}{45}{0.003}{0.2*.7}{black}{};
			\draw[red] (0,0) -- (6*.0027,.455*.7) 
			-- (18*.0027,.95*.7)
			-- (180*.0027,4.6*.7)
			-- (199*.0027,4.96*.7)
			-- (635*.0027,5.28*.7)
			-- (668*.0027,5.31*.7)
			-- (745*.0027,4.99*.7)
			-- (1148*.0027,3.0*.7)
			-- (1275*.0027,2.08*.7)
			-- (1337*.0027,1.4*.7)
			-- (1362*.0027,.89*.7)
			-- (1369*.0027,.68*.7)
			node[right] { $C$};
			\node[circle, fill=black, scale=0.5] at (933*.0027,188*3*.7/403+215*5.03*.7/403) {}{};
			\node[circle, fill=black, scale=0.5, label=above right:{ $V_1$}] at (933*.0027,0) {}{};
			\draw[dashed] (933*.0027,0) -- (933*.0027,188*3*.7/403+215*5.03*.7/403);
			\node[circle, fill=blue, scale=0.5] at (745*.0027,5.03*.7) {}{};
			\node[circle, fill=blue, scale=0.5, label=below: {\color{blue} $g_1$}] at (745*.0027,0) {}{};
			\node[circle, fill=blue, scale=0.5, label=left: {\color{blue} $C(g_1)$}] at (0,5.03*.7) {}{};
			\draw[blue,dashed] (745*.0027,0) -- (745*.0027,5.03*.7);
			\draw[blue,dashed] (0,5.03*.7) -- (745*.0027,5.03*.7);
			\node[circle, fill=blue, scale=0.5] at (1148*.0027,3*.7) {}{};
			\node[circle, fill=blue, scale=0.5, label=below: {\color{blue} $d_1$}] at (1148*.0027,0) {}{};
			\node[circle, fill=blue, scale=0.5, label=left: {\color{blue} $C(d_1)$}] at (0,3*.7) {}{};
			\draw[blue,dashed] (1148*.0027,0) -- (1148*.0027,3*.7);
			\draw[blue,dashed] (0,3*.7) -- (1148*.0027,3*.7);
			\draw [thin, draw=gray, ->] (0,0) -- (1400*.0027,0);
			\draw [thin, draw=gray, ->] (0,0) -- (0,5.5*.7);
			\node[circle, fill=black, scale=0.5, label=below: $T$] at (1369*.0027,0) {}{};
			\draw[very thick, draw=blue] (0,0) -- (1369*.0027,0);
		\end{tikzpicture}
		\begin{tikzpicture}
			\pgfmathsetseed{101101}
			\BrownianMotion{0}{0}{100}{0.003}{0.2*.7}{black}{};
			\BrownianMotion{101*.0027}{2.5*.7}{123}{0.003}{0.2*.7}{black}{};
			\BrownianMotion{225*.0027}{2.9*.7}{54}{0.003}{0.2*.7}{black}{};
			\BrownianMotion{280*.0027}{2.4*.7}{172}{0.003}{0.2*.7}{black}{};
			\BrownianMotion{453*.0027}{0.8*.7}{27}{0.003}{0.2*.7}{black}{};
			\BrownianMotion{481*.0027}{2.0*.7}{80}{0.003}{0.2*.7}{black}{};
			\BrownianMotion{562*.0027}{1.3*.7}{69}{0.003}{0.2*.7}{black}{};
			\BrownianMotion{632*.0027}{4.6*.7}{113}{0.003}{0.2*.7}{black}{};
			\phantom{
				\BrownianMotion{632*.0027}{4.6*.7}{13}{0.003}{0.2*.7}{black}{};
				\BrownianMotion{759*.0027}{0.4*.7}{208}{0.003}{0.2*.7}{black}{};
				\BrownianMotion{968*.0027}{0.8*.7}{67}{0.003}{0.2*.7}{black}{};
				\BrownianMotion{1036*.0027}{1.8*.7}{113}{0.003}{0.2*.7}{black}{};
			}
			\BrownianMotion{1149*.0027}{3.0*.7}{19}{0.003}{0.2*.7}{black}{};
			\BrownianMotion{1169*.0027}{0.1*.7}{164}{0.003}{0.2*.7}{black}{};
			\BrownianMotion{1324*.0027}{1.0*.7}{45}{0.003}{0.2*.7}{black}{};
			\draw[red] (0,0) -- (6*.0027,.455*.7) 
			-- (18*.0027,.95*.7)
			-- (180*.0027,4.6*.7)
			-- (199*.0027,4.96*.7)
			-- (635*.0027,5.28*.7)
			-- (668*.0027,5.31*.7)
			--(745*.0027,4.99*.7){};
			\draw[red] (1148*.0027,3.0*.7)
			-- (1275*.0027,2.08*.7)
			-- (1337*.0027,1.4*.7)
			-- (1362*.0027,.89*.7)
			-- (1369*.0027,.68*.7)
			node[right] { $C$};
			\node[circle, fill=black, scale=0.5] at (433*.0027,202*5.33*.7/436+234*5.02*.7/436) {}{};
			\node[circle, fill=black, scale=0.5, label=below:{ $V_2$}] at (433*.0027,0) {}{};
			\draw[dashed] (433*.0027,0) -- (433*.0027,202*5.33*.7/436+234*5.02*.7/436);
			\node[circle, fill=blue, scale=0.5] at (199*.0027,5.02*.7) {}{};
			\node[circle, fill=blue, scale=0.5, label=below: {\color{blue} $g_2$}] at (199*.0027,0) {}{};
			\node[circle, fill=blue, scale=0.5, label=below left: {\color{blue} $C(g_2)$}] at (0,5.02*.7) {}{};
			\draw[blue,dashed] (199*.0027,0) -- (199*.0027,5.02*.7);
			\draw[blue,dashed] (0,5.02*.7) -- (199*.0027,5.02*.7);
			\node[circle, fill=blue, scale=0.5] at (635*.0027,5.33*.7) {}{};
			\node[circle, fill=blue, scale=0.5, label=below: {\color{blue} $d_2$}] at (635*.0027,0) {}{};
			\node[circle, fill=blue, scale=0.5, label=left: {\color{blue} $C(d_2)$}] at (0,5.33*.7) {}{};
			\draw[blue,dashed] (635*.0027,0) -- (635*.0027,5.33*.7);
			\draw[blue,dashed] (0,5.33*.7) -- (635*.0027,5.33*.7);
			\draw [thin, draw=gray, ->] (0,0) -- (1400*.0027,0);
			\draw [thin, draw=gray, ->] (0,0) -- (0,5.5*.7);
			\draw [very thick, draw=blue] (0,0) -- (745*.0027,0);
			\draw [very thick, draw=blue] (1148*.0027,0) -- (1369*.0027,0);
			\node[circle, fill=black, scale=0.5, label=below: $T$] at (1369*.0027,0) {}{};
		\end{tikzpicture}
		\begin{tikzpicture}
			\pgfmathsetseed{101101}
			\BrownianMotion{0}{0}{100}{0.003}{0.2*.7}{black}{};
			\BrownianMotion{101*.0027}{2.5*.7}{98}{0.003}{0.2*.7}{black}{};
			\phantom{
				\BrownianMotion{101*.0027}{2.5*.7}{25}{0.003}{0.2*.7}{black}{};
				\BrownianMotion{225*.0027}{2.9*.7}{54}{0.003}{0.2*.7}{black}{};
				\BrownianMotion{280*.0027}{2.4*.7}{172}{0.003}{0.2*.7}{black}{};
				\BrownianMotion{453*.0027}{0.8*.7}{27}{0.003}{0.2*.7}{black}{};
				\BrownianMotion{481*.0027}{2.0*.7}{80}{0.003}{0.2*.7}{black}{};
				\BrownianMotion{562*.0027}{1.3*.7}{69}{0.003}{0.2*.7}{black}{};
				\BrownianMotion{632*.0027}{4.6*.7}{2}{0.003}{0.2*.7}{black}{};
			}
			\BrownianMotion{635*.0027}{4.95*.7}{111}{0.003}{0.2*.7}{black}{};
			\phantom{
				\BrownianMotion{632*.0027}{4.6*.7}{13}{0.003}{0.2*.7}{black}{};
				\BrownianMotion{759*.0027}{0.4*.7}{208}{0.003}{0.2*.7}{black}{};
				\BrownianMotion{968*.0027}{0.8*.7}{67}{0.003}{0.2*.7}{black}{};
				\BrownianMotion{1036*.0027}{1.8*.7}{113}{0.003}{0.2*.7}{black}{};
			}
			\BrownianMotion{1149*.0027}{3.0*.7}{19}{0.003}{0.2*.7}{black}{};
			\BrownianMotion{1169*.0027}{0.1*.7}{164}{0.003}{0.2*.7}{black}{};
			\BrownianMotion{1324*.0027}{1.0*.7}{45}{0.003}{0.2*.7}{black}{};
			\draw[red] (0,0) -- (6*.0027,.455*.7) 
			-- (18*.0027,.95*.7)
			-- (180*.0027,4.6*.7)
			-- (199*.0027,4.96*.7){};
			\draw[red] (635*.0027,5.28*.7)
			-- (668*.0027,5.31*.7)
			--(745*.0027,4.99*.7){};
			\draw[red] (1148*.0027,3.0*.7)
			-- (1275*.0027,2.08*.7)
			-- (1337*.0027,1.4*.7)
			-- (1362*.0027,.89*.7)
			-- (1369*.0027,.68*.7)
			node[right] { $C$};
			\node[circle, fill=black, scale=0.5] at (150*.0027,30*.95*.7/162+132*4.6*.7/162) {}{};
			\node[circle, fill=black, scale=0.5, label=below:{ $V_3$}] at (150*.0027,0) {}{};
			\draw[dashed] (150*.0027,0) -- (150*.0027,30*.95*.7/162+132*4.6*.7/162);
			\node[circle, fill=blue, scale=0.5] at (18*.0027,.95*.7) {}{};
			\node[circle, fill=blue, scale=0.5, label=below:{\color{blue} $g_3$}] at (18*.0027,0) {}{};
			\node[circle, fill=blue, scale=0.5, label=above left:{\color{blue} $C(g_3)$}] at (0,.95*.7) {}{};
			\draw[blue,dashed] (18*.0027,0) -- (18*.0027,.95*.7);
			\draw[blue,dashed] (0,.95) -- (18*.0027,.95*.7);
			\node[circle, fill=blue, scale=0.5] at (180*.0027,4.6*.7) {}{};
			\node[circle, fill=blue, scale=0.5, label=below right:{\color{blue}  $d_3$}] at (180*.0027,0) {}{};
			\node[circle, fill=blue, scale=0.5, label=below left:{\color{blue} $C(d_3)$}] at (0,4.6*.7) {}{};
			\draw[blue,dashed] (180*.0027,0) -- (180*.0027,4.6*.7);
			\draw[blue,dashed] (0,4.6*.7) -- (180*.0027,4.6*.7);
			\draw [thin, draw=gray, ->] (0,0) -- (1400*.0027,0);
			\draw [thin, draw=gray, ->] (0,0) -- (0,5.5*.7);
			\draw [very thick, draw=blue] (0,0) -- (199*.0027,0);
			\draw [very thick, draw=blue] (635*.0027,0) -- (745*.0027,0);
			\draw [very thick, draw=blue] (1148*.0027,0) -- (1369*.0027,0);
			\node[circle, fill=black, scale=0.5, label=below: $T$] at (1369*.0027,0) {}{};
		\end{tikzpicture}
		\caption{\label{fig:FacesCM} 
			Randomly selecting the first three faces of the concave majorant $C$ of $X$ (the smallest concave function dominating the path of $X$) in a size-biased way. The total length of the thick blue segment(s) on the abscissa equal the stick remainders $L_0=T$, $L_1=T-\ell_1$ and $L_2=T-\ell_1-\ell_2$, respectively, where $\ell_1=d_1-g_1$ and $\ell_2=d_2-g_2$.  The independent random variables $V_1,V_2,V_3$ are uniform on the sets $[0,T]$, $[0,T]\setminus(g_1,d_1)$, $[0,T]\setminus\bigcup_{i=1}^2(g_i,d_i)$, respectively. The interval $(g_i,d_i)$, is determined by the edge of the concave majorant which includes $V_i$. By \cite[\S 4.1]{LevySupSim}, this procedure yields a stick-breaking process $\ell$ and, conditionally given $\ell$, the increments
			$C(d_i)-C(g_i)$ are independent with the same law as $X_t$ at $t=\ell_i$, i.e., $C(d_i)-C(g_i)\,\eqd\, \ell_i^{1/\alpha}S_i$.}
	\end{figure}
	
	The procedure starts by constructing a random sequence of disjoint sub-intervals of the time interval $[0,T]$ which will cover it geometrically fast. This is usually called a stick-breaking process: $\ell=(\ell_i)_{i\geq1}$ on 
	the interval $[0,T]$. That is, based on the i.i.d. standard uniform random variables 
	$U_i\sim\U(0,1)$, define $L_0:=T$ and for each $i\in\N$,
	$L_i:=L_{i-1}U_k$ and $\ell_i=L_{i-1}-L_i=L_{i-1}(1-U_i)=T(1-U_i)\prod_{j=1}^{i-1}U_j$. It is not difficult to see that $\sum_{i=1}^\infty \ell_i=T$ and that for any $p>0$,  $\mathbb{E}[\ell_{{i}}^{p}]=T^p\left(1+p\right)^{{-i}}$. That is, the convergence of the total length of the sequence of disjoint intervals $\bigcup_{j=1}^i[L_{j-1},L_j]$ to $T$ is geometrically fast.

	Now, we define the Chambers-Mallows-Stuck approximation for stable laws. We need to define a sequence of them in order to approximate $X_{\pm}$. For an independent 
	i.i.d. sequence of stable random variables $(S_i)_{i\geq1}$ with parameters 
	$(\alpha,\rho)$ (i.e. $S_i \eqd X_1$). When $\alpha\ne 1$, these stable 
	random variables can be represented as (see~\cite{Weron}) 
	\[
	S_i = E_i^{1-1/\alpha}G_i
	\qquad\text{and}\qquad 
	G_i=g(V_i),\qquad i\in\N,
	\]
	for i.i.d. exponential random variables $(E_i)_{i\geq1}$ with unit mean  
	independent of the i.i.d. $\U(-\frac{\pi}{2},\frac{\pi}{2})$ random variables 
	$(V_i)_{i\geq1}$ and function  
	\begin{equation}
		\label{eq:g-function}
		g(x):=\frac{\sin\big(\alpha\big(x+\omega\big)\big)}
		{\cos^{1/\alpha}(x)
			\cos^{1-1/\alpha}\big((1-\alpha)x-\alpha\omega\big)},
		\qquad x\in\Big(-\frac{\pi}{2},\frac{\pi}{2}\Big),
	\end{equation}
	where $\omega:=\pi(\rho-\tfrac{1}{2})$. 
	Note that indeed $\p(S_i>0)=\rho$. We assume that all the above random 
	variables are defined on a probability space $(\Omega,\F,\p)$. These 
	random elements and the coupling in~\cite[\S 4.1]{LevySupSim} provide an 
	almost sure representation for $(X_T,\ov{X}_T)$:
	\begin{equation}
		\label{eq:X_pm}
		\ov X_T=X_+
		\qquad\text{and}\qquad
		X_T=X_+ -X_-,
		\qquad\text{where}\qquad 
		X_\pm:=\sum_{i=1}^\infty\ell_i^{1/\alpha}[S_i]^\pm.
	\end{equation}
	The series in the definitions of $X_+$ and $X_-$ have non-negative terms and converge almost surely by the equalities in~\eqref{eq:X_pm}. Note again, that the convergence in the above infinite sum is ``geometrically fast'' due to the behavior of the stick breaking process.

	As stated in the Introduction, we will base our finite dimensional integration by parts formulas using the exponential random variables $E_i$ which characterize heuristically the ``length'' of the stable random variable $S_i$ while the ``oscillating'' part $G_i$ will not be used in order to determine the regularity of the law and therefore all calculations will be conditioned on this random sequence. This key observation makes possible our analysis. 
	
	In order to build approximations of the above random 
	variables on finite dimensional spaces with smooth laws, we will truncate the infinite sums up to the $n$-th term. With this in mind and in order to preserve the existence of densities, we replace the remainder with $a_n\eta_{\pm}$ as follows: let $(a_n)_{n\in\N}$ be a positive and strictly decreasing 
	sequence defined as $ a_n:=T^{1/\alpha} \kappa^n$  with $ \kappa\in (0,1)$. 
	Therefore $a_n\downarrow0$ as $n\to\infty$. 
	The random variables $\eta_{\pm}$ are exponentially distributed with unit mean 
	independent of each other and of every other random variable. 
	With these elements we define the $n$-th approximation to $\chi=(X_+,X_-)$ 
	as $\chi_n=(X_{+,n},X_{-,n})$, $n\in\N$ given by 
	\begin{equation}
		\label{eq:xpm}
		X_{\pm,n}
		:=
		\sum_{i=1}^n
		\ell_i^{1/\alpha}[S_i]^\pm+a_n \eta_\pm^{1-1/\alpha}
		=\sum_{i=1}^n
		\ell_i^{1/\alpha}E_i^{1-1/\alpha}[G_i]^\pm+a_n \eta_\pm^{1-1/\alpha}.
	\end{equation}
	In the case $n=0$, we define $X_{\pm,0}:=0$.

	We introduce the following assumption which will 
	be valid throughout the paper.
	\begin{assumption*}[A-$\kappa$]
		\label{asm:seq_decay}
		The constant $\kappa\in (0,1)$ in $a_n=T^{1/\alpha} \kappa^n$ satisfies 
		$\kappa^\alpha\ge\rho\vee (1-\rho)$. 
	\end{assumption*}
	This assumption is crucial in order to obtain good positive and negative moment estimates for $X_{\pm,n}$ within the bounds allowed by stable laws (see Lemma~\ref{lem:inv-mom}).

	For any $m\in\N$, $n\in\N\cup\{\infty\}$ and $A\subset\R^m$, 
	let $C_b^{n}(A)$ be the set of bounded and $n$-times continuously 
	differentiable functions $f:\R^m\to\R$ on the open set $A$ and whose 
	derivatives of order at most $n$ are all bounded. Furthermore for 
	$f\in C_b^{1}(\R^2)$ we denote the partial derivatives with respect to the 
	first and second component by $\partial_+ f$ and $\partial_- f$, respectively. 
	\subsection{Sequential integration by parts formulae via a multilevel method}
	\label{subsec:Ibpf}
	
	In order to state the finite dimensional Ibpf based on exponential random variables, we will use a derivative operator notation with 
	respect to this set of random variables. Thus, for any random variable 
	$F=f(\vartheta, \mathcal{K})$, where $f$ is differentiable in the first 
	component and the random variable $\vartheta$ is independent of the 
	random element $\mathcal{K}$, the derivative $\partial_\vartheta[\cdot]$
	is well-defined and given by the formula
	$\partial_\vartheta[F]=\partial_\vartheta f(\vartheta, \mathcal{K})$.
	As stated above, the random variables 
	$\{E_i, U_i, V_i, \eta_{\pm}; i\in\N\}$ are independent (i.e. the joint law 
	is a product measure), making the derivatives in the following lemma well-defined. We start stating some basic properties of the differential operator which will be used in our arguments.
	
	
	\begin{lem} 
		\label{lem:3}
		For any $m\in\N$, define the differential operators
		\begin{equation}
			\label{eq:D}
			\D^\pm_m
			:=
			\eta_\pm\partial_{\eta_\pm}
			+\sum_{i=1}^m E_i \1_{\{[G_i]^\pm>0\}}\partial_{E_i}.
		\end{equation}
		Then for any function $f:\R_+\to\R_+$ and $p\in\R\setminus\{0\}$ we have 
		\begin{equation}
			\label{eq:reg}
			\begin{split}
				E_i\partial_{E_i}[X_{\pm,n}]
				&=(1- 1/\alpha)\ell_i^{1/\alpha}
				E_i^{1-1/\alpha}[G_i]^\pm\1_{\{i\le n\}},
				\hspace{12mm} k\in\N,\\
				\D^\pm_m\big[\big(X^p_{\pm,n},f(X_{\mp,n})\big)\big]
				&=( 1 - 1/\alpha)\big(pX_{\pm,n}^{p},0\big),
				\hspace{17.5mm} m\ge n\ge 1.\\
			\end{split}
		\end{equation}
	\end{lem}
	
	\begin{proof}
		The first two identities follow easily. For the third identity, note that 
		$X_{\pm,n}>0$ a.s. and thus, its reciprocal and any of its powers are always well 
		defined real numbers. The other identities follow from the first one and the 
		corresponding formula for $\eta_\pm\partial_{\eta_\pm}[X_{\pm,n}]$. 
	\end{proof}
	\begin{rem}
	    \begin{enumerate}[leftmargin=1cm]
	        \item 	The identity 
	$\D^\pm_m X^p_{\pm,n} = ( 1- 1/\alpha)pX_{\pm,n}^{p}$, 
	$m\geq n\geq1$, in~\eqref{eq:reg} reveals a crucial regenerative property 
	of $X_{\pm,n}$ with respect to the operator $\D^\pm_m$ (like the fact that in classical calculus the derivative of the exponential function is itself). In fact, this is the 
	main motivation behind the definition of $\D^\pm_m$. 
	This regenerative structure relies heavily on the particular dependence of 
	$X_{\pm,n}$ with respect to $S_i$ and $E_i$, $ i\in\{1,\ldots,n\}$. 
	\item The indicators $1_{\{[G_i ]^{\pm} >0\}}$ in the definition of $\D^\pm_m$ ensure that 
	when applied to $f(\chi_n)$, only one of the partial derivatives of $f$ appear due to \eqref{eq:reg}
	(see~\eqref{eq:1.3} below).
	    \end{enumerate}
	\end{rem}

	Now, we introduce the space of 
	smooth random variables. Given any metric space $S$, define the space of 
	real-valued bounded and continuous on $(0,\infty)^m\times S$ that are 
	$C_b^\infty$ in its first $m$ components 
	\[ 
	\s_\infty((0,\infty)^m,S)
	:= \Big\{\phi:(0,\infty)^m\times S\to\R
	;\enskip \phi\text{ is continuous},
	\enskip \phi(\cdot,s)\in C_b^\infty((0,\infty)^m;\R),\ \forall s\in S
	\Big\}.
	\] 
	Then we define 
	\[ 
	\s_m(\Omega)
	:=
	\big\{\Phi\in L^0(\Omega): 
	\enskip\exists \phi(\cdot,\vartheta)\in 
	\s_\infty((0,\infty)^{3m+2}, S), \enskip
	\Phi=\phi(\mathcal{E}_m,\mathcal{U}_m,\mathcal{V}_m,
	\eta_+,\eta_-,\vartheta)\big\},
	\]
	where  
	$\mathcal{E}_m:=(E_1,\ldots,E_m)$, $\mathcal{U}_m:=(U_1,\ldots,U_m)$, 
	$\mathcal{V}_m:=(V_1,\ldots,V_m)$ and
	$\vartheta$ is any random element in some metric space $S$ independent 
	of $(\mathcal{E}_m,\mathcal{U}_m,\mathcal{V}_m,\eta_+,\eta_-)$. {For instance, if the random variable $ \Phi $ is a function of $ ( \mathcal{E}_\infty, \mathcal{U}_\infty,\mathcal{V}_\infty)$, we say that $ \Phi\in \s_m(\Omega) $ if the property defining this set is satisfied with $ \vartheta = ((E_{m+1},E_{m+2},\ldots),(U_{m+1},U_{m+2},\ldots),(V_{m+1},V_{m+2},\ldots)) $ representing all the random variables with indices larger than $ m$.}
	We describe now the following finite dimensional Ibpf for a fixed approximation parameter 
	$n$. Recall that $\chi_n=(X_{+,n},X_{-,n})$ for $n\in\N$. 
	\begin{prop}
		\label{prop:IBP_n}	
		Fix $n,m\in\N$ with $m\ge n$. Then for any  $\Phi\in\s_m(\Omega)$ and 
		$f\in C_b^1((\varepsilon,\infty)^2)$, 
		\begin{equation}\label{eq:IBP_n}\begin{split}
				& \E[\partial_{\pm} f(\chi_n)\Phi]
				=\E[f(\chi_n)H^\pm_{n,m}(\Phi)],
				\quad\text{where}\\
				& H^\pm_{n,m}(\Phi)
				:=\frac{1}{X_{\pm,n}}
				\frac{\alpha}{\alpha-1}
				\Big(\Big(\eta_\pm-\frac{1}{\alpha}
				+\sum_{i=1}^m(E_{i}-1)\1_{\{[G_i]^\pm>0\}}\Big)\Phi
				-\D^\pm_m[\Phi]\Big)\in \s_m(\Omega).
		\end{split}\end{equation}
	\end{prop}
	
	
	\begin{proof}
		Note that $[x]^\pm>0$ if and only if $\pm x>0$. 
		The chain rule for derivatives and~\eqref{eq:reg} yield 
		\begin{equation}
			\label{eq:1.3}
			\D^\pm_m[f(\chi_n)]
			=\partial_\pm f(\chi_n)\D^\pm_m[X_{\pm,n}]
			=( 1- 1/\alpha)\partial_\pm f(\chi_n) X_{\pm,n}.
		\end{equation}
		
		Denote $\widetilde\partial_\vartheta[Y]:=Y-\partial_\vartheta[Y]$. 
		Let $\eta$ be an exponential random variable with unit mean. Observe that if $\Lambda_i:=h_i(\eta)$ for some 
		$h_i\in\s_\infty((0,\infty);\R)$, $i\in\{1,2\}$, then the classical Ibpf 
		(with respect to the density of $\eta$) gives 
		\begin{equation}
			\label{eq:IBP_n_aux}
			\begin{split}
				\E[\Lambda_1 \eta\partial_{\eta}[\Lambda_2]]
				&=\E[\partial_{\eta}[\Lambda_1\Lambda_2\eta]
				-\Lambda_2\partial_{\eta}[\Lambda_1\eta]]=
				\E[\Lambda_1\Lambda_2\eta
				-\Lambda_2\partial_{\eta}[\Lambda_1\eta]]
				=\E[\Lambda_2\widetilde{\partial}_{\eta}[\Lambda_1\eta]].
			\end{split}
		\end{equation}
		
		Integration by parts with respect to $\eta_\pm $ 
		and $E_k$ for each $i\leq n$ gives, by~\eqref{eq:reg},~\eqref{eq:1.3} 
		and~\eqref{eq:IBP_n_aux},
		\begin{equation}
			\label{eq:IBP_n_proof}
			\begin{split}
				\E\left[\partial_\pm f(\chi_n)\Phi|\F_{-E}\right]
				&=\frac{\alpha}{\alpha-1}\E\bigg[
				\frac{\Phi}{X_{\pm,n}}\D^\pm_m[f(\chi_n)]\bigg|\F_{-E}\bigg]\\
				&=\frac{\alpha}{\alpha-1}\E\bigg[
				f(\chi_n)
				\bigg(\widetilde\partial_{\eta_\pm}\bigg[\frac{\Phi\eta_\pm}{X_{\pm,n}}\bigg]
				+\sum_{i=1}^m\widetilde\partial_{E_i}\bigg[
				\frac{\Phi E_i\1_{\{[G_i]^\pm>0\}}}{X_{\pm,n}}\bigg]\bigg)
				\bigg|\F_{-E}\bigg]\\
				&=\E[f(\chi_n)H^\pm_{n,m}(\Phi)|\F_{-E}].
		\end{split}\end{equation}
		Here we have denoted by $\F_{-E}$ the $\sigma$-algebra generated by all 
		but the exponential random variables $\eta_+$, $\eta_-$ and $E_{i}$, 
		$i\in\N$ which are used in the integration-by-parts. Taking expectations 
		in~\eqref{eq:IBP_n_proof} completes the proof. 
	\end{proof}

	\begin{rem}
		\begin{enumerate}[leftmargin=1cm]
			\item
			Observe that the role of $\varepsilon$ in the previous result is to ensure 
			that the expectation on the right-hand side in~\eqref{eq:IBP_n} is finite 
			(by making the quotient $f(\chi_n)/X_{\pm,n}$ bounded).
			\item Recall that exponential laws are discontinuous at zero. Still, in the above Ibpf, these boundary terms do no appear. This is due to the factors $E_i\partial_{E_i}$ and $\eta_{\pm}\partial_{\eta_\pm}$ which appear in the definition of $\mathcal{D}^{\pm}_m$ in \eqref{eq:D}. In exchange, one has $X_{\pm,n}$ in the denominator of the expression for $H^{\pm}_{n,m}(\Phi)$.
		\end{enumerate}
	\end{rem}
	As 
	$H^\pm_{n,m}(\Phi)\in\s_m(\Omega)$ 
	for any $\Phi\in\s_m(\Omega)$, $m\ge n$, we inductively define the sequence of 
	operators $\{H_{n,m}^{\pm,k}(\cdot)\}_{k\in\N}$ for every $n,m\in\N$ such 
	that $m\ge n$ as
	\[
	H_{n,m}^{\pm,k+1}(\Phi)
	:=H_{n,m}^\pm (H_{n,m}^{\pm,k}(\Phi))
	\qquad\text{for }k\geq 0,\text{ where}\qquad
	H_{n,m}^{\pm,0}(\Phi)
	:=\Phi.
	\]

		Let us state some basic properties of the weights $H^\pm_{n,m}(\Phi) $. 
			\begin{lem}
		\begin{enumerate}[leftmargin=1cm]
			\item If $\alpha\ne 1$ and $\Phi$ do not depend on 
			$\mathcal{E}_m$ or $\eta_\pm$, then we have $\D^\pm_m[\Phi]=0$ and 
			hence $H^\pm_{n,m}(\Phi)=H^\pm_{n,m}(1)\Phi$.
			\item The operators $H_{n,m}^{\pm,k}(\cdot)$ and $H_{n,m}^{\mp,j}(\cdot)$ 
			commute. 
		\end{enumerate}
		\end{lem}
	These iterated operators are useful in order to define the multiple Ibpf formulas for the limit random variables in combination with the so-called Multi level Monte Carlo method which can be interpreted as an interpolation formula which uses approximations in order to describe the behavior of the limit. This is done in the next result.
	
	\begin{thm}
		\label{thm:series}
		Let $\Phi\in \s_n(\Omega)$ for all $n\in\mathbb{N}$. For any $n\geq 1$, 
		$k_+,k_-\ge 0$ and $f\in C_b^{k_++k_-}([\varepsilon,\infty)^2)$ we have 
		\begin{align}
			\label{eq:IBP}
			\E\big[\partial_+^{k_+}\partial_-^{k_-} f(\chi)\Phi\big]
			=&\mathbb{E}\left[	\langle f,\Phi\rangle_n^{k_+,k_-}\right]\\
			\notag\langle f,\Phi\rangle_n^{k_+,k_-}:=&f(\chi_{n})H_{n,n}^{+,k_+}\big(H_{n,n}^{-,k_-}(\Phi))\\&+\sum_{i=n}^\infty\left(f(\chi_{i+1})H_{i+1,i+1}^{+,k_+}\big(H_{i+1,i+1}^{-,k_-}(\Phi)\big)
			- f(\chi_{i})H_{i,i+1}^{+,k_+}\big(H_{i,i+1}^{-,k_-}(\Phi))\right)\label{eq:ip}
		\end{align}
	\end{thm}
	
	\begin{proof} 
		Note that $\E[\widetilde f(\chi_{n})]\to\E[\widetilde f(\chi)]$ as $ n\to \infty $ for any 
		bounded and continuous function $\widetilde f$ since $\chi_n\to\chi$ 
		a.s. Recall that $\partial_+^{k_+}\partial_-^{k_-}f$ is continuous and 
		bounded. By telescoping we find 
		\[
		\E[\partial_+^{k_+}\partial_-^{k_-} f(\chi)\Phi]
		=\E[\partial_+^{k_+}\partial_-^{k_-}f(\chi_{n})\Phi]
		+ \E\left[\sum_{i=n}^\infty(\partial_+^{k_+}\partial_-^{k_-}f(\chi_{i+1})
		-\partial_+^{k_+}\partial_-^{k_-}f(\chi_{i}))\Phi\right ].
		\]
		
		The first term equals $\E[f(\chi_{n})H^{+,k_+}_{n,n}(H^{-,k_-}_{n,n}(\Phi))]$ by 
		Proposition~\ref{prop:IBP_n}. Applying Proposition~\ref{prop:IBP_n} 
		again shows that each term in the above sum
		equals its corresponding term in~\eqref{eq:ip}, yielding~\eqref{eq:IBP}. 
	\end{proof}
	
	It is clear that iterations of $H_{n,m}^\pm$ have long and complex explicit 
	expressions. In particular, the remaining goal is to find proper bounds for the iterated operators which appear in formula \eqref{eq:IBP}.

	In order to complete the arguments for our main proofs we will need that the infinite sum appearing in \eqref{eq:IBP} converges absolutely. Furthermore bounding this sum becomes important in obtaining upper bounds for the joint density and its derivatives. This is all done at once in the next lemma. Its proof is technical but only uses basic algebra and moments of the random variables in involved in Section \ref{sec:3.1}.

	We are interested in the 
	explicit decay rate of the terms in the sum of Theorem~\ref{thm:series} for 
	a special class of functions $f$ related to the distribution of $\chi$. 
	This description will then be used to finally prove 
	Theorem~\ref{thm:density_bound}. More precisely, given some measurable 
	and bounded $h:\R_+^2\to\R$ and $x_+,x_->0$, we will consider the function 
	$f$ given by 
	\begin{equation}
		\label{eq:area-function}
		f(x,y):=\int_0^x\int_0^y h(x',y')\1_{\{x'>x_+,y'>x_-\}}dy' dx',
		\qquad x,y\in\R_+.
	\end{equation}
	
	We are interested in such class of functions since the particular choice 
	$h = 1$ yields $\E[\partial_+\partial_-f(\chi)] = \p(X_+>x_+,X_->x_-)$. 
	Note also that for a general $h$, the inequality $|f(x,y)|\le \|h\|_\infty xy$ 
	holds for any $x,y\in\R_+$, where $\|h\|_\infty:=\sup_{x,y\in\R_+}|h(x,y)|$. 
	We denote by $\mathcal{A}(K,x_+,x_-)$, $K>0$, the class of functions $f$ 
	satisfying~\eqref{eq:area-function} for some measurable 
	function $h:\R_+^2\to\R$ with $\|h\|_\infty\le K$. 
	
	We denote the random variables arising in $ \langle f,\Phi\rangle_n^{k_+,k_-} $ of Theorem~\ref{thm:series} by 
	\begin{align}
		\label{eq:Theta}
		\begin{split}
			\Theta_{n,m}^f\equiv \Theta_{n,m}^{f,\Phi}(k_+,k_-)
			:=&f(\chi_n)H_{n,m}^{+,k_+}\big(H_{n,m}^{-,k_-}(\Phi)\big),
			\enskip\text{ for $m\ge n$, and }\\
			\enskip
			\widetilde\Theta^f_n\equiv \widetilde\Theta^{f,\Phi}_n(k_+,k_-) :=&\Theta^f_{n+1,n+1}(k_+,k_-)-\Theta^f_{n,n+1}(k_+,k_-).\
		\end{split}
	\end{align}
	We will drop $\Phi$ and or $ (k_+,k_-) $ from the notation if it is well understood from the context. The following key result provides bounds on moments. 
	
	\begin{prop}
		\label{prop:Theta_bound}
		Let $\kappa\in (0,1)$ be as in Assumption~(\nameref{asm:seq_decay}). 
		Fix any $p\ge 1$, $k_\pm\ge 2$ and $\alpha'\in[0,\alpha)$. 
		Given some $\phi\in C_b^{k_++k_-}(\R^2)$, define $\Phi:=\phi(\chi)$. 
		Let the variables $\Theta_{n,m}^f$ and $\widetilde\Theta_{n}^f$ be given 
		by~\eqref{eq:Theta}, then the following statements hold. 
		\item[\normalfont(a)] For $s:=p\wedge \alpha'$ there is a constant 
		$C>0$ such that for any $K,T,x_+,x_->0$ and $m\ge n$:
		\begin{align}
			\label{eq:int_density-decay}
			\E\bigg[\sup_{f\in\mathcal{A}(K,x_+,x_-)}|\widetilde\Theta^f_n|^p\bigg]
			&\le CK^p\frac{T^{2\frac{\alpha'}{\alpha}}
				\big(\big(1+\tfrac{s}{\alpha}\big)^{-n}+\kappa^{ns}\big)n^{p'}}
			{x_+^{p(k_+-1)+\alpha'}x_-^{p(k_--1)+\alpha'}},\\
			\label{eq:int_density-decay2}
			\E\bigg[\sup_{f\in\mathcal{A}(K,x_+,x_-)}|\Theta_{n,m}^f|^p\bigg]
			&\le CK^p\frac{T^{2\frac{\alpha'}{\alpha}}m^{p'}}
			{x_+^{p(k_+-1)+\alpha'}x_-^{p(k_--1)+\alpha'}},
		\end{align}
		where $p'=\max\{p(k_++k_-),1\} + [\alpha'-1]^+ + [\alpha'-s-1]^+$. 
		\item[\normalfont(b)] Consider any 
		$u\in(0,(\alpha-\alpha')(\rho\wedge(1-\rho))/p)$ and let 
		$p'=p(k_++k_-)$, then for some $C>0$ and all $K,T,x_+,x_->0$ and $m\ge n$, 
		the following inequalities hold 
		\begin{align}
			\label{eq:int_density-decay3}
			\E\bigg[\sup_{f\in\mathcal{A}(K,x_+,x_-)}|\widetilde\Theta^f_n|^p\bigg]
			&\le  CK^p\frac{T^{-\frac{\alpha'}{\alpha}}
				\big(\big(1+\tfrac{pu}{\alpha}\big)^{-n}+\kappa^{npu}\big)
				n^{p'}}
			{x_+^{p(k_+-1)-\alpha'\rho}x_-^{p(k_--1)-\alpha'(1-\rho)}},\\
			\label{eq:int_density-decay4}
			\E\bigg[\sup_{f\in\mathcal{A}(K,x_+,x_-)}|\Theta_{n,m}^f|^p\bigg]
			&\le CK^p\frac{T^{-\frac{\alpha'}{\alpha}}m^{p'}}
			{x_+^{p(k_+-1)-\alpha'\rho}x_-^{p(k_--1)-\alpha'(1-\rho)}}.
		\end{align}
		\item[\normalfont(c)]
		Consider any $u\in(0,(\alpha-\alpha')(\rho\wedge(1-\rho))/p)$ and let 
		$p'=p(k_++k_-)$, then for some $C>0$ and all $K,T,x_+,x_->0$ and $m\ge n$, 
		the following inequalities hold 
		\begin{align}
			\label{eq:int_density-decay5}
			\E\bigg[\sup_{f\in\mathcal{A}(K,x_+,x_-)}|\widetilde\Theta^f_n|^p\bigg]
			&\le CK^p\frac{\big(\big(1+\tfrac{pu}{\alpha}\big)^{-n}+\kappa^{npu}\big)
				n^{p'}}{x_+^{p(k_+-1)}x_-^{p(k_--1)}}
			\min\bigg\{\frac{T^{\frac{\alpha'}{\alpha}(1-\rho)}}
			{x_+^{-\alpha'\rho}x_-^{\alpha'}},
			\frac{T^{\frac{\alpha'}{\alpha}\rho}}
			{x_+^{\alpha'}x_-^{-\alpha'(1-\rho)}}\bigg\},\\
			\label{eq:int_density-decay6}
			\E\bigg[\sup_{f\in\mathcal{A}(K,x_+,x_-)}|\Theta_{n,m}^f|^p\bigg]
			&\le CK^p\frac{m^{p'}}{x_+^{p(k_+-1)}x_-^{p(k_--1)}}
			\min\bigg\{\frac{T^{\frac{\alpha'}{\alpha}(1-\rho)}}
			{x_+^{-\alpha'\rho}x_-^{\alpha'}},
			\frac{T^{\frac{\alpha'}{\alpha}\rho}}
			{x_+^{\alpha'}x_-^{-\alpha'(1-\rho)}}\bigg\}.
		\end{align}
	\end{prop}
	
	\begin{rem}
		\label{rem:Theta_bound}
		{\normalfont(i)} Clearly the above inequalities imply the absolute convergence of the infinite sum in \eqref{eq:IBP}.\\
		{\normalfont(ii)} The reason for the different cases is that we will use (a) when $x_+$ and $x_-$ both take large values, part (b) when 
		they are both small and part (c) for the mixed case in which $x_+$ is small 
		and $x_-$ is large or vice versa, cf. Figure~\ref{fig:regions}.
		
	\end{rem}
	The proof of this key technical result is given in Section \ref{sec:mom}.
	With these preparations, now we are ready to give the proof of our main result.
	\section{Proof of Theorem~\ref{thm:density_bound}}
	\label{sec:4}

	In the present subsection we will prove Theorem~\ref{thm:density_bound}. 	
	We will follow the structure presented in the proof of Theorem 2.1.4 in \cite{nual:06}. In fact, consider a test function $ f\in C^\infty_b(\mathbb{R}^2) $ then a similar  representation as \eqref{eq:area-function} gives for $ F(x,y):=[x-x_+]^+[y-x_-]^+ $
	\begin{align*}
		f(x_+,x_-)=\int_{\mathbb{R}_+^2}F(x',y')\partial_+\partial_-f(x',y')dy'dx'.
	\end{align*} Next, using  Theorem \ref{thm:series} and Fubini theorem with $ \hat{F}(x',y')= [x'-X_+]^+[y'-X_-]^+$, we obtain
	\begin{align*}
		\mathbb{E}\left[\partial_+\partial_-f(\chi)\right]=\mathbb{E}\left[\langle f,1\rangle _n^{1,1}\right]=\int_{\mathbb{R}_+^2} \partial_+\partial_-f(x',y')\mathbb{E}\left[\langle \hat{F}(x',y'),1\rangle _n^{1,1}\right] dy' dx'.
	\end{align*}
	This readily implies that the density of $ \chi $ at $ (x',y') $ exists and can be expressed as $\mathbb{E}\left[\langle \hat{F}(x',y'),1\rangle _n^{1,1}\right]  $.
	
	In a similar fashion, one considers for $ k_+,k_-\geq 1 $
	\begin{align*}
		\mathbb{E}\left[\partial^{k_+}_+\partial^{k_-}_-f(\chi)\right]=\mathbb{E}\left[\langle f,1\rangle _n^{k_+,k_-}\right]=\int_{\mathbb{R}_+^2} \partial_+\partial_-f(x',y')\mathbb{E}\left[\langle \hat{F}(x',y'),1\rangle _n^{k_+,k_-}\right] dy' dx'.
	\end{align*}
	From here, one obtains the regularity of the law of $ \chi $.
	The next step, is to obtain the upper bound for $ \mathbb{E}\left[\langle \hat{F}(x_+,x_-),1\rangle _n^{k_+,k_-}\right] $.
	That is, our goal is to prove 
	\begin{equation}
		\label{eq:10.4}
		\begin{split}
			\left	| \mathbb{E}\left[\langle \hat{F}(x_+,x_-),1\rangle _n^{k_+,k_-}\right]\right |
			&\le Cx_+^{-k_-}x_-^{-k_+}\\
			&\hspace{-8mm}
			\times\min\big\{
			T^{2\frac{\alpha'}{\alpha}}x_+^{-\alpha'}x_-^{-\alpha'},
			T^{\frac{\alpha'}{\alpha}\rho}x_+^{-\alpha'}x_-^{\alpha'(1-\rho)},
			T^{\frac{\alpha'}{\alpha}(1-\rho)}x_+^{\alpha'\rho}x_-^{-\alpha'},
			T^{-\frac{\alpha'}{\alpha}}x_+^{\alpha'\rho}x_-^{\alpha'(1-\rho)}\big\}.
		\end{split}
	\end{equation}
	In fact, the bounds follow from  Proposition~\ref{prop:Theta_bound} (a)--(c) (with $p=1$). We use part (a) when $x_+$ and $x_-$ both take large values, part (b) when 
	they are both small and part (c) for the mixed case in which $x_+$ is small 
	and $x_-$ is large or vice versa, cf. Figure~\ref{fig:regions}.
	Each application of Proposition~\ref{prop:Theta_bound} yields summable 
	upper bounds on the summands of the series defined by $ \langle \hat{F}(x_+,x_-),1\rangle _n^{k_+,k_-} $. The minimum in~\eqref{eq:10.4} is the 
	smallest sum of these upper bounds as a function of $(x_+,x_-)$ and $T$. 
	
	Observe that the derivatives of $F$ in Theorem~\ref{thm:density_bound} 
	can be expressed in terms of the derivatives of $G(x_+,x_-):=\p(X_+ > x_+, X_- > x_-)$ as follows: the linear 
	transformation $(X_T,\ov X_T)\mapsto(\ov X_T,\ov X_T - X_T)$ yields 
	$\partial_x\partial_yF(x,y)=\partial_+\partial_-G(y,y-x)$ for $y> x\vee 0$  
	and thus 
	\[
	\partial_x^n\partial_y^mF(x,y)
	=(-1)^{n-1}\sum_{i=0}^{m-1}
	\binom{m-1}{i}\partial_+^{m-i}\partial_-^{n+i}G(y,y-x).
	\]
	
	Therefore,~\eqref{eq:10.4} gives~\eqref{eq:density_at_infinity} as follows:
	\[
	\begin{split}
		|\partial_x^n\partial_y^mF(x,y)|
		&\le
		\sum_{i=0}^{m-1}\binom{m-1}{i}|\partial_+^{m-i}\partial_-^{n+i}G(y,y-x)|\\
		&\le 
		\sum_{i=0}^{m-1}\binom{m-1}{i}Cy^{i-m}(y-x)^{-n-i}
		\min\{f_{\alpha'}^{00}(x,y),f_{\alpha'}^{01}(x,y),
		f_{\alpha'}^{10}(x,y),f_{\alpha'}^{11}(x,y)\}\\
		&= Cy^{-m}(y-x)^{1-n-m}(2y-x)^{m-1}
		\min\{f_{\alpha'}^{00}(x,y),f_{\alpha'}^{01}(x,y),
		f_{\alpha'}^{10}(x,y),f_{\alpha'}^{11}(x,y)\}.\qedhere
	\end{split}
	\]
	\hfill$\Box$
	
	\section{Technical Results}
	\subsection{Upper bounds on the Ibpf} 
	\label{subsec:indicators}
	In this section, we study the upper bounds in the technical Proposition 
	\ref{prop:Theta_bound}. It is the key result in order to obtain Theorem~\ref{thm:density_bound}. 	We start with some basic properties for the operator $ H $ which are useful for bounding $ 	\Theta_{n,m}^f $ and $ \widetilde\Theta^f_n $.
	
	For any $m\in\N$, define 
	\begin{equation*}
		\Sigma^\pm_m
		:=\eta_\pm  + \sum_{i=1}^m E_i\1_{\{[G_i]^\pm>0\}}
		\qquad\text{and}\qquad
		\sigma^\pm_m
		:=1 + \sum_{i=1}^m \1_{\{[G_i]^\pm>0\}}.
	\end{equation*}
	With this notation and in this case, 
	we may rewrite for $\Phi\in\mathbb{S}_m(\Omega)$, 
	\begin{equation}
		\label{eq:1.8}
		\begin{split}
			H_{n,m}^\pm(\Phi)
			&=\frac{\alpha/(\alpha-1)}{X_{\pm,n}}
			\Big(\Big(\Sigma^\pm_m-\sigma^\pm_m+1-\frac{1}{\alpha}\Big)\Phi
			-\D^\pm_m[\Phi]\Big),\\
			\D^\pm_m[\Sigma^\pm_m]
			&=\Sigma^\pm_m,
			\qquad \D^\pm_m[\sigma^\pm_m]=0.
		\end{split}
	\end{equation}
	
	\begin{lem}
		\label{lem:simplify_HPhi}
		Fix any $k_\pm\ge 0$ and suppose $\Phi:=\phi(\chi)$ for 
		some $\phi\in C_b^{k_++k_-}(A)$ with $A\subset\R_+^2$. 
		Then for any $m>n$, we have 
		\begin{equation}
			\label{eq:H_indep_n}
			H_{n,m}^{+,k_+}(H_{n,m}^{-,k_-}(\Phi))X_{+,n}^{k_+}X_{-,n}^{k_-}
			=H_{n+1,m}^{+,k_+}(H_{n+1,m}^{-,k_-}(\Phi))X_{+,n+1}^{k_+}X_{-,n+1}^{k_-}.
		\end{equation}
		Moreover, if we set 
		\begin{align}
			\label{def:Z}
			Z_m:=\Sigma_{+,m}+\Sigma_{-,m}= \eta_++\eta_-+\sum_{i=1}^m E_i, \quad m\in\N,
		\end{align} 
		then there is a bivariate polynomial $p^\phi_{k_+,k_-}(\cdot,\cdot)$ of 
		degree at most $k_++k_-$ whose coefficients do not depend on $n$ or $m$, 
		such that 
		\begin{equation}
			\label{eq:H(Phi)_bound}
			|H_{n,m}^{+,k_+}(H_{n,m}^{-,k_-}(\Phi))X_{+,n}^{k_+}X_{-,n}^{k_-}|
			\le \1_{\{\chi\in A\}}p_{k_+,k_-}^\phi(Z_m,m),
			\qquad\text{for all }m\ge n.
		\end{equation} 
	\end{lem}
	
	\begin{proof}
		The proof is simple: we only need to expand the formula for 
		$H_{n,m}^{+,k_+}(H_{n,m}^{-,k_-}(\Phi))$ and then uniformly bound all the 
		derivatives of $\phi$  by the same constant. 
		
		Recalling that  
		$\D_m^\pm[(\Sigma^\pm_m,X_{\pm,n}^{-p})]=
		(\Sigma^\pm_m,(1/\alpha-1)pX_{\pm,n}^{-p})$ and 
		$\D_m^\mp[(\Sigma^\pm_m,\sigma^\pm_m,\sigma^\mp_m,X_{\pm,n}^{-p})]=0$ 
		for $p>0$, we deduce that an iteration of~\eqref{eq:1.8} yields 
		$X_{+,n}^{-k_+}X_{-,n}^{-k_-}$ multiplied by a polynomial of degree $k_+$ in 
		$\Sigma_m^+$ and $\sigma_m^+$. Its coefficients  are themselves polynomials 
		of degree $k_-$ in $\Sigma_m^-$ and $\sigma_m^-$ multiplied by a linear 
		combination of the derivatives $\partial_+^{j_+}\partial_-^{j_-}\phi(\chi)$ for 
		$j_\pm\le k_\pm$. This directly implies~\eqref{eq:H_indep_n}. Since those 
		derivatives are bounded and we have the a.s. bounds 
		$\Sigma_m^\pm\le Z_m$ 
		and $\sigma_m^\pm\le m$, we may bound the entire expression by a 
		constant (independent of $n$ and $m$) multiplied by a polynomial of 
		degree $k_++k_-$ in $Z_m$ and $m$. This completes the proof in this case.
		
	\end{proof}

	\subsection{ Proof of Proposition~\ref{prop:Theta_bound}, Part I: Interpolation inequalities}
	\label{subsec:alg}
	As we stated previously the proof of the technical Proposition~\ref{prop:Theta_bound} is self-contained and it is divided in two parts. In a first part, we mainly use basic inequalities which will depend on powers of $ X_{\pm,n} $, $ Z_m $, $ \ell_n $, $ \eta_{\pm} $ and $\Delta_{\pm,n}:=X_{\pm,n}-X_{\pm,n-1}  $. These moments properties are studied later in Section \ref{sec:mom}. We assume those results and give the proof of this Proposition here.
	\begin{proof}[Proof of Proposition~\ref{prop:Theta_bound}]
			In the estimates that follow, we will make repeated use of the following inequalities: 
		\begin{align}
			\label{eq:8.1}
			\bigg|\sum_{i=1}^k x_i\bigg|^q\le k^{[q-1]^+}\sum_{i=1}^k|x_i|^q,
			\qquad \text{for any }q>0 \text{ and } x_i\in\R,
		\end{align}
		which follows from the concavity of $x\mapsto x^q$ if $q\le 1$ and Jensen's 
		inequality if $q>1$. Moreover, we frequently apply the following basic
		interpolating inequalities: $\1_{\{y>x\}}\le y^vx^{-v}$ for all $ v\geq 0 $ where we interpret the upper bound as $ 1 $ if $ v=0 $. Also, if $y,z\ge 0$ then for all $ r\in [0,1] $
		\begin{align}
			\label{eq:geom_comb}
			y\wedge z\le& y^r z^{1-r}\\
			y\vee z\geq& y^rz^{1-r}.
		\end{align}
		Define $(m_{\pm,n},M_{\pm,n})
		:=(X_{\pm,n}\wedge X_{\pm,n+1},X_{\pm,n}\vee X_{\pm,n+1})$ then 	$m_{\pm,n}=X_{\pm,n+1}\wedge X_{\pm,n}\ge \kappa X_{\pm,n}$ since 
		$X_{\pm,n+1}\ge \kappa X_{\pm,n}$. Similarly, $M_{\pm,n}\le \kappa^{-1}X_{\pm,n+1}$.

		\textbf{Part (a).} We will proceed in three steps. 
		Step I) is also used in the proofs of (b) and (c).
		
		I) Recall the definition $Z_m$ in \eqref{def:Z} and consider 
		the polynomial $p^\phi_{k_+,k_-}$ from Lemma~\ref{lem:simplify_HPhi}. 
		According to Lemma~\ref{lem:simplify_HPhi} with $A=\R_+^2$, we have for $\widetilde{f}(x,y):= f(x,y)/(x^{k_+}y^{k_-})$
		\begin{align}
			\nonumber
			\big|\widetilde\Theta^f_n\big|^p
			&=\big|f(\chi_{n+1})H_{n+1,n+1}^{+,k_+}\big(H_{n+1,n+1}^{-,k_-}(\Phi)\big)
			- f(\chi_n)H_{n,n+1}^{+,k_+}\big(H_{n,n+1}^{-,k_-}(\Phi)\big)\big|^p\\
			\label{eq:p-indicator}
			&=
			\left|\frac{f(\chi_{n+1})}{X_{+,n+1}^{k_+}X_{-,n+1}^{k_-}}
			-\frac{f(\chi_n)}{X_{+,n}^{k_+}X_{-,n}^{k_-}}\right|^p
			\big|H_{n,n+1}^{+,k_+}\big(H_{n,n+1}^{-,k_-}(\Phi)\big)
			X_{+,n}^{k_+}X_{-,n}^{k_-}\big|^p\\
			\nonumber
			&\le 
			\left|\widetilde{f}(\chi_{n+1})-\widetilde{f}(\chi_{n})\right|^p
			p^\phi_{k_+,k_-}(Z_{n+1},n+1)^p.
		\end{align}
		The goal for the rest of the proof is to provide algebraic inequalities for the above expression which depend explicitly on powers of $ \Delta_{\pm,n} $, $X_{\pm,n} $ and $ Z_{n+1} $. Through these expressions, we will later  		
		show that, in expectation, the first factor in the last line decays 
		geometrically in $n$ while the second factor has polynomial growth in $n$.
		
		II) Next, we obtain an upper bound for the modulus of continuity of the map 
		$\widetilde{f}$ which appears in \eqref{eq:p-indicator} and where $f$ is given 
		in~\eqref{eq:area-function}. This map is absolutely continuous with respect 
		to Lebesgue measure and thus a.e. differentiable with
		\begin{align*}
			|\partial_+\widetilde{f}(x,y)|
			&= \1_{\{x>x_+,y>x_-\}}\bigg|
			\frac{\partial_+f(x,y)}{x^{k_+}y^{k_-}}
			-\frac{k_+ f(x,y)}{x^{k_++1}y^{k_-}}\bigg|
			\le \1_{\{x>x_+,y>x_-\}}c_1x^{-k_+}y^{1-k_-},\\
			|\partial_-\widetilde{f}(x,y)|
			&= \1_{\{x>x_+,y>x_-\}}\bigg|
			\frac{\partial_- f(x,y)}{x^{k_+}y^{k_-}}
			-\frac{k_-f(x,y)}{x^{k_+}y^{k_-+1}}\bigg|
			\le \1_{\{x>x_+,y>x_-\}}c_1x^{1-k_+}y^{-k_-},
		\end{align*}
		where $c_1:=(k_++1)(k_-+1)\|h\|_\infty$. Then, for any $x,x',y,y'\in\R_+$, denote 
		$(m_x,M_x):=(x\wedge x',x\vee x')$ and $(m_y,M_y):=(y\wedge y',y\vee y')$ and 
		observe: 
		\begin{align}
			\nonumber
			|\widetilde{f}(x,y)-\widetilde{f}(x',y')|
			&=\bigg|\int_{x'}^x\partial_+\widetilde{f}(z,y)dz 
			+ \int_{y'}^y\partial_+\widetilde{f}(x',z)dz\bigg|\\
			\label{eq:Dg_bound}
			&\le \frac{\1_{\{M_x>x_+,M_y>x_-\}}c_1|x-x'|}
			{(m_x\vee x_+)^{k_+}(m_y\vee x_-)^{k_--1}}
			+ \frac{\1_{\{M_x>x_+,M_y>x_-\}}c_1|y-y'|}
			{(m_x\vee x_+)^{k_+-1}(m_y\vee x_-)^{k_-}}\\
			\label{eq:Dg_bound0}
			&\le \1_{\{M_x>x_+,M_y>x_-\}}\frac{c_1}{x_+^{k_+}x_-^{k_-}}
			\big(|x-x'|x_- + |y-y'|x_+\big).
		\end{align}
		Note that in the inequality in~\eqref{eq:Dg_bound} we used that 
		$k_+,k_-\ge 2$ and that the support of $g$ is contained in 
		$[x_+,\infty)\times[x_-,\infty)$. Moreover, since $f$ 
		in~\eqref{eq:area-function} satisfies $|f(x,y)|\le \|h\|_\infty xy$, we have 
		$|\widetilde f(x,y)|\le \|h\|_\infty x^{1-k_+}y^{1-k_-}$. 
		Hence, for any $x,x',y,y'\in\R_+$ we have 
		\begin{equation}
			\label{eq:Dg_bound2}
			\begin{split}
				|\widetilde f(x,y)-\widetilde f(x',y')|
				&\le \1_{\{M_x>x_+,M_y>x_-\}}\sup_{z>m_x,w>m_y}|\widetilde f(z,w)|\\
				&\le \1_{\{M_x>x_+,M_y>x_-\}}
				c_2(m_x\vee x_+)^{1-k_+}(m_y\vee x_-)^{1-k_-},
			\end{split}
		\end{equation}
		where $c_2:=2\|h\|_\infty$. Typically, for each maximum in the 
		denominator, we use a geometric mixing of its arguments. 
		
		III)  Now, with the above bound we will show that the upper bound for $ \widetilde\Theta^f_n $ depends on moments of basic random variables. 
		Recall that $s=p\wedge\alpha'$. Applying~\eqref{eq:8.1} (with $q=s/p$) 
		and~\eqref{eq:geom_comb} (with $r=s/p$) to the minimum of the two bounds obtained in~\eqref{eq:Dg_bound0} 
		and~\eqref{eq:Dg_bound2}  in the form $\eqref{eq:Dg_bound0}^{s/p} \eqref{eq:Dg_bound2}^{1-s/p} $ and using $x_+\le m_x\vee x_+$ and 
		$x_-\le m_y\vee x_-$ yields: for any $x,x',y,y'\in\R_+$ the 
		following inequality holds,  
		\begin{equation*}
			|\widetilde f(x,y)-\widetilde f(x',y')|
			\le \frac{\1_{\{M_x>x_+,M_y>x_-\}}c_3^{1/p}}{x_+^{k_+-1+s/p}x_-^{k_--1+s/p}}
			\big(|x-x'|^{s/p}x_-^{s/p} + |y-y'|^{s/p}x_+^{s/p}\big),
		\end{equation*}
		where $c_3:=c_1^{s}c_2^{p-s}$. This interpolation method is used in all cases with different combinations.

		Then~\eqref{eq:8.1} gives 
		\begin{equation*}
			|\widetilde f(\chi_{n+1})-\widetilde f(\chi_n)|^p
			\le \frac{\1_{\{M_{+,n}>x_+,M_{-,n}>x_-\}} 2^{p-1}c_3}
			{x_+^{p(k_+-1)+s}x_-^{p(k_--1)+s}}
			\big(|\Delta_{+,n+1}|^s x_-^s + |\Delta_{-,n+1}|^s x_+^s\big).
		\end{equation*}
		Applying the inequality
		$\1_{\{M_{\pm,n}>x_\pm\}}\le x_\pm^{-v}M^v_{\pm,n}$, for $v=\alpha'-s\geq 0$
		and~\eqref{eq:p-indicator} we obtain 
		\[
		\begin{split}
			|\widetilde\Theta^f_n|^p
			&\le \frac{2^{p-1}c_3p^\phi_{k_+,k_-}(Z_{n+1},n+1)^p}
			{x_+^{p(k_+-1)+\alpha'}x_-^{p(k_--1)+\alpha'}}
			\big(|\Delta_{+,n+1}|^s M_{+,n}^{\alpha'-s}M_{-,n}^{\alpha'} 
			+|\Delta_{-,n+1}|^s M_{-,n}^{\alpha'-s}M_{+,n}^{\alpha'}\big)\\
			&\le \frac{2^{p-1}c_3p^\phi_{k_+,k_-}(Z_{n+1},n+1)^p}
			{\kappa^{2\alpha'-s}x_+^{p(k_+-1)+\alpha'}x_-^{p(k_--1)+\alpha'}}
			\big(|\Delta_{+,n+1}|^s X_{+,n+1}^{\alpha'-s}X_{-,n+1}^{\alpha'} 
			+|\Delta_{-,n+1}|^s X_{-,n+1}^{\alpha'-s}X_{+,n+1}^{\alpha'}\big),
		\end{split}
		\]
		where the second inequality follows from the fact that 
		$M_{\pm,n}\le \kappa^{-1} X_{\pm,n+1}$. Finally, as $\alpha'<\alpha$, 
		Lemma~\ref{lem:bound_X+X-Z} gives~\eqref{eq:int_density-decay}. 
		
		To prove the second statement in (a), we proceed as before. We start by 
		using the inequality $|\widetilde f(\chi_n)|^p\le \1_{\{X_{+,n}>x_+,X_{-,n}>x_-\}}
		\|h\|_\infty^p x_+^{p(1-k_+)}x_-^{p(1-k_-)}$ and the 
		bound $\1_{\{X_{\pm,n}>x_\pm\}}\le X_{\pm,n}^{\alpha'}x_\pm^{-\alpha'}$. 
		An application of Lemma~\ref{lem:bound_X+X-Z} then 
		yields~\eqref{eq:int_density-decay2}. 
		
		\textbf{Part (b).} Let $c_4:=2^{1-1/p}c_1^u c_2^{1-u}$ where $u\in[0,1]$ is 
		given in the statement. Applying~\eqref{eq:geom_comb} (with $r=u$) 
		and~\eqref{eq:8.1} (with $q=p$) to the minimum of \eqref{eq:Dg_bound} and  \eqref{eq:Dg_bound2} in the form $ \eqref{eq:Dg_bound} ^u \eqref{eq:Dg_bound2}^{1-u}$
		yields 
		\begin{align}
			\nonumber
			|\widetilde f(x,y)-\widetilde f(x',y')|^p
			&\le \1_{\{M_x>x_+,M_y>x_-\}}c_4^p
			\frac{|x-x'|^{pu}/(m_x\vee x_+)^{pu}+|y-y'|^{pu}/(m_y\vee x_-)^{pu}}
			{(m_x\vee x_+)^{p(k_+-1)}(m_y\vee x_-)^{p(k_--1)}}\\
			\label{eq:Dg_bound5}
			&\le \1_{\{M_x>x_+,M_y>x_-\}}c_4^p
			\frac{|x-x'|^{pu}/m_x^{pu}+|y-y'|^{pu}/m_y^{pu}}
			{(m_x\vee x_+)^{p(k_+-1)}(m_y\vee x_-)^{p(k_--1)}}.
		\end{align}
		
		By~\eqref{eq:geom_comb} we have $m_x\vee x_+\ge m_x^r x_+^{1-r}$ and 
		$m_y\vee x_-\ge m_y^{r'} x_-^{1-r'}$ for any $r,r'\in[0,1]$. Since 
		$\alpha'<\alpha\le 1/[\rho\vee(1-\rho)]$, we choose 
		$r=\alpha'\rho/[p(k_+-1)]$ and $r'=\alpha'(1-\rho)/[p(k_--1)]$. Applying 
		these interpolating inequalities to~\eqref{eq:Dg_bound5} and combining them 
		with~\eqref{eq:p-indicator} gives 
		\[
		\begin{split}
			\big|\widetilde\Theta^f_n\big|^p
			&\le \frac{c_4^pp^\phi_{k_+,k_-}(Z_{n+1},n+1)^p}
			{x_+^{p(k_+-1)-\alpha'\rho}x_-^{p(k_--1)-\alpha'(1-\rho)}}
			\bigg(\frac{|\Delta_{+,n+1}|^{pu}}
			{m_{+,n}^{\alpha'\rho+pu}m_{-,n}^{\alpha'(1-\rho)}}
			+\frac{|\Delta_{-,n+1}|^{pu}}
			{m_{+,n}^{\alpha'\rho}m_{-,n}^{\alpha'(1-\rho)+pu}}\bigg),\\
			&\le \frac{c_4^pp^\phi_{k_+,k_-}(Z_{n+1},n+1)^p}
			{\kappa^{\alpha'+pu}	
				x_+^{p(k_+-1)-\alpha'\rho}x_-^{p(k_--1)-\alpha'(1-\rho)}}
			\bigg(\frac{|\Delta_{+,n+1}|^{pu}}
			{X_{+,n}^{\alpha'\rho+pu}X_{-,n}^{\alpha'(1-\rho)}}
			+\frac{|\Delta_{-,n+1}|^{pu}}
			{X_{+,n}^{\alpha'\rho}X_{-,n}^{\alpha'(1-\rho)+pu}}\bigg),
		\end{split}
		\]
		where we used the restriction that 
		$m_{\pm,n}\ge \kappa X_{\pm,n}$. 
		Moreover, as $u\in (0,(\alpha-\alpha')(\rho\wedge(1-\rho))/p)$, we have $\alpha'\rho+pu<\alpha\rho$ and 
		$\alpha'(1-\rho)+pu<\alpha(1-\rho)$. Hence, applying 
		Lemma~\ref{lem:inv_bound_X+X-Z} gives~\eqref{eq:int_density-decay3}. 
		
		The proof of~\eqref{eq:int_density-decay4} is analogous to that 
		of~\eqref{eq:int_density-decay3}. Indeed, using~\eqref{eq:geom_comb} 
		and the inequality $|\widetilde f(\chi_n)|
		\le \|h\|_\infty (X_{+,n}\vee x_+)^{1-k_+}(X_{-,n}\vee x_-)^{1-k_-}$ we obtain 
		\[
		|\widetilde f(\chi_n)|^p
		\le \|h\|_\infty^p 
		x_+^{p(1-k_+) + \alpha'\rho}x_-^{p(1-k_-) + \alpha'(1-\rho)}
		X_{+,n}^{-\alpha'\rho}X_{-,n}^{-\alpha'(1-\rho)}. 
		\]
		The inequality~\eqref{eq:int_density-decay4} then follows from 
		Lemma~\ref{lem:inv_bound_X+X-Z}, completing the proof of~(b). 
		
		\textbf{Part (c).} We will only prove the bound for the first argument of
		the minimum in the right hand side of ~\eqref{eq:int_density-decay5} 
		and~\eqref{eq:int_density-decay6}; the other case is analogous. 
		We proceed as in (b): using~\eqref{eq:geom_comb}, \eqref{eq:p-indicator}, 
		\eqref{eq:Dg_bound5} but instead of the interpolating inequality using $ r $, we use the bound 
		$\1_{\{M_{+,n}>x_+\}}\le M_{+,n}^vx_+^{-v}$, for any $v=\alpha'-pu,\alpha'\geq 0$), we obtain 
		\[
		\begin{split}
			\big|\widetilde\Theta^f_n\big|^p
			&\le \frac{c_4^pp^\phi_{k_+,k_-}(Z_{n+1},n+1)^p}
			{x_+^{p(k_+-1)+\alpha'}x_-^{p(k_--1)-\alpha'(1-\rho)}}
			\bigg(\frac{|\Delta_{+,n+1}|^{pu}M_{+,n}^{\alpha'-pu}}
			{m_{-,n}^{\alpha'(1-\rho)}}
			+\frac{|\Delta_{-,n+1}|^{pu}M_{+,n}^{\alpha'}}
			{m_{-,n}^{\alpha'(1-\rho)+pu}}\bigg)\\
			&\le \frac{c_4^pp^\phi_{k_+,k_-}(Z_{n+1},n+1)^p}
			{\kappa^{\alpha'(2-\rho)+pu}
				x_+^{p(k_+-1)+\alpha'}x_-^{p(k_--1)-\alpha'(1-\rho)}}
			\bigg(\frac{|\Delta_{+,n+1}|^{pu}X_{+,n+1}^{\alpha'-pu}}
			{X_{-,n}^{\alpha'(1-\rho)}}
			+\frac{|\Delta_{-,n+1}|^{pu}X_{+,n+1}^{\alpha'}}
			{X_{-,n}^{\alpha'(1-\rho)+pu}}\bigg),
		\end{split}
		\]
		where we used the fact that $M_{\pm,n}\le \kappa^{-1}X_{\pm,n+1}$ and 
		$m_{\pm,n}\ge \kappa X_{\pm,n}$. 
		An application of Lemma~\ref{lem:inv_bound_X+X-Z} then 
		gives~\eqref{eq:int_density-decay5}. 
		
		Using the inequality $|\widetilde f(\chi_n)|\le \1_{\{X_{+,n}>x_+,X_{-,n}>x_-\}}
		\|h\|_\infty (X_{+,n}\vee x_+)^{1-k_+}(X_{-,n}\vee x_-)^{1-k_-}$, 
		\eqref{eq:geom_comb} and the bound 
		$\1_{\{X_{+,n}>x_+\}}\le X_{+,n}^{\alpha'} x_+^{-\alpha'}$, we obtain 
		\[
		\big|\Theta^f_{n,m}\big|^p
		\le \frac{\|h\|_\infty^p p^\phi_{k_+,k_-}(Z_m,m)^p}
		{x_+^{p(k_+-1)+\alpha'}x_-^{p(k_--1)-\alpha'(1-\rho)}}
		\frac{X_{+,n}^{\alpha'}}
		{X_{-,n}^{\alpha'(1-\rho)}}.
		\] 
		which yields~\eqref{eq:int_density-decay6} by 
		Lemma~\ref{lem:inv_bound_X+X-Z}, completing the proof of the proposition. 
	\end{proof}
	\begin{rem}Analyzing the  above proof, we can see the interpolation method at work here. In fact, the estimates of Proposition~\ref{prop:Theta_bound}, 
		one may say that all polynomial terms in $n$ arise due to the polynomial 
		growth of $H_{n,m}^{\pm,k_{\pm}}$ (see~\eqref{eq:H(Phi)_bound} 
		in Lemma~\ref{lem:simplify_HPhi}), through the term $ p^\phi_{k_+,k_-}(Z_m,m)^p $ which appears in the upper bounds. On the other hand, the geometrically decreasing 
		terms are produced by the exponentially fast decay of the differences 
		$\Delta_{\pm,n}:=X_{\pm,n}-X_{\pm,n-1}$ in $\widetilde\Theta^f_n$. We stress that another achievement of the interpolation method is that the moment estimates of 
		Proposition~\ref{prop:Theta_bound} hold for any $p\ge 1$.
		\end{rem}
	
	\section{Proof of Proposition~\ref{prop:Theta_bound}, Part II: The moment bounds }
	\label{sec:mom}
	
	
	
	In this section, we state the explicit moment estimates for the quantities that 
	appear in the weights of the multiple Ibpf of Theorem~\ref{thm:series}. 
	These bounds were the last step in the proof of Proposition~\ref{prop:Theta_bound} 
	above. The proofs of these lemmas in this section, are independent of everything that 
	have preceded them. In order to obtain near optimal bounds in 
	Theorem~\ref{thm:density_bound}, we first study the growth of the moments 
	of $X_{\pm,n}^p$ for $p$ arbitrarily close to $\alpha$ in 
	Lemmas~\ref{lem:bound_X+X-Z}, \ref{lem:inv-mom} 
	and~\ref{lem:inv_bound_X+X-Z}. Since the 
	$\alpha$-moment of the stable law does not exist, the bounds in these lemmas cannot be 
	obtained e.g. via H\"older's inequality. Their proofs consist of a direct, but
	very careful, analysis of the corresponding expectations. 
	
	There are two types of bounds according to whether they involve positive 
	or negative moments of $ X_{\pm,n} $. They correspond to the behavior at infinity or 
	at zero in the estimates that we obtain in Theorem~\ref{thm:density_bound} as can be deduced from the proof of Proposition~\ref{prop:Theta_bound}. 
	Throughout the present section we use the notation from 
	Subsection~\ref{sec:3.1}. 
	In particular, recall the definition of $ Z_m $ in \eqref{def:Z}
	and Assumption~(\nameref{asm:seq_decay}):
	$\kappa^\alpha\in [\rho\vee(1-\rho),1)$. 
	Explicit constants in the results in this section can be recovered 
	from the proofs. 
	
		We begin by recalling the Mellin transform of a 
	stable random variable (see~\cite[Thm~2.6.3]{MR854867}) 
	\begin{equation*}
		\E[S_1^p\1_{\{S_1>0\}}]
		= \rho\frac{\Gamma(1+p)\Gamma(1-p/\alpha)}
		{\Gamma(1+p\rho)\Gamma(1-p\rho)},\qquad p\in(-1,\alpha).
	\end{equation*}
	When $\alpha\ne 1$, by the independence $E_i\indep G_i$ we deduce that, 
	for any $p\in[0,\alpha)$,
	\begin{equation}
		\label{eq:Mellin_G}
		\E[G_i^p\1_{\{G_i>0\}}]
		=\frac{\E[S_i^p\1_{\{S_i>0\}}]}{\E\big[E_i^{p(1-1/\alpha)}\big]}
		=\frac{\rho\Gamma(1+p)\Gamma(1-p/\alpha)}
		{\Gamma(1+p\rho)\Gamma(1-p\rho)\Gamma(p(1-1/\alpha)+1)}.
	\end{equation}
	Finally, we recall that $\E[E_1^p]=\Gamma(1+p)$ is finite if and only if 
	$p>-1$.

		\subsection{Positive moments}
	\begin{lem}
		\label{lem:bound_X+X-Z}
		Let $p,q,s\ge 0$ satisfy $q\le p<\alpha$. Then, there exists a constant $C>0$ 
		such that for any $m\ge n$ and $T>0$ we have 
		\begin{align}
			\label{eq:13.1}
			\E\big[X_{\pm,n}^{p-q}X_{\mp,n}^p |\Delta_{\pm,n}|^q Z_m^s\big]
			\le CT^{2\frac{p}{\alpha}}
			\Big(\big(1+\tfrac{q}{\alpha}\big)^{-n}
			+\kappa^{qn}\Big)m^{[p-1]^+ + [p-q-1]^+ +s}.
		\end{align}
	\end{lem}

	Before giving the proof of the lemma, we give some remarks:
	\begin{rem} 	{\normalfont(i)} Note that the exponent in $E^{1-1/\alpha}$ appearing in $ X_{\pm,n} $
		changes sign when $\alpha\in (0,1)$ and $\alpha\in (1,2)$. For this reason, 
		most of the proofs for estimating the above bounds of moments will have to be 
		done in three separate cases: $\alpha\in(0,1)$, $\alpha=1$ and 
		$\alpha\in(1,2)$. This makes the proofs slight long because some inequalities change depending on the above cases. \\
		{\normalfont(ii)} Note that due to the scaling property of the stick breaking process and $ a_n $ the factor of $ T^{2\frac{p}{\alpha}} $ is easily obtained. In fact, we will assume, without loss of generality, in all proofs in this section that $ T=1 $. In the Lemma statements, we have left the dependence on $ T $ and in some major points of the proof too. In a first reading, one may assume always that $ T=1 $.\\
		{\normalfont(iii)} We will consider in all proofs only one combination of $ \pm $ signs. The other case follows mutatis mutandis.
		\end{rem}


	\begin{proof}[Proof of Lemma~\ref{lem:bound_X+X-Z}]
	We first make a number of reductions that simplify the proof. 
	We will assume $p,q>0$. The remaining cases (when at least one of the two 
	parameters is zero) follow similarly by ignoring the corresponding terms 
	in the calculations. 
	
	Let $c=2^{[p-1]^++[p-q-1]^++[q-1]^+}$ and use~\eqref{eq:8.1} to obtain 
	\begin{align}
		\nonumber
		X_{+,n}^{p-q}X_{-,n}^p|\Delta_{+,n}|^q
		&\le c\bigg(\bigg(\sum_{i=1}^n \ell_i^{1/\alpha}[S_i]^+\bigg)^{p-q}
		+ a_{n}^{p-q}\eta_+^{p-q}\bigg)
		\bigg(\bigg(\sum_{i=1}^n \ell_i^{1/\alpha}[S_i]^-\bigg)^{p}
		+ a_{n}^{p}\eta_-^{p}\bigg)\\
		&\qquad\times
		\big((\ell_n^{1/\alpha}[S_n]^+)^q + a_{n-1}^q\eta_+^q\big).
		\label{eq:12.1}
	\end{align}
	
	Our goal is now to provide an upper bound for the expectation of the right 
	hand side of the above inequality multiplied by $Z_n^s$. This leads to eight 
	terms which must be treated individually to show that their expectations 
	decay exponentially at least as a polynomial (in $n$) multiple of 
	$a_{n-1}^q$ or $\E[\ell_n^{q/\alpha}]=(1+q/\alpha)^{-n}$. 
	We treat the hardest term in~\eqref{eq:12.1}; which contains the product of sums of $ [S_i]^{\pm} $. The other terms are 
	easier to treat as we remark at the end of the proof. Therefore we will consider, for $r\in\{0,q\}$ 
	
	\begin{align}
		\label{eq:13.0}
		A
		:=\E\bigg[\bigg(\sum_{i=1}^n 
		\ell_i^{1/\alpha}[G_i]^+c_i\bigg)^{p-q}
		\bigg(\sum_{j=1}^n \ell_j^{1/\alpha}[G_j]^-c_j\bigg)^{p}
		(\ell_n^{1/\alpha}[G_n]^+c_n)^r
		\bigg| \F_{-E}\bigg]Z_m^s,
		\qquad r\in\{0,q\},
	\end{align}
	where $c_i=E_i^{1-1/\alpha}$ and $\F_{-E}=\sigma(\ell_i,G_i;i\in\N)$. 
	We estimate~\eqref{eq:13.0} in steps:
	
	I) In this step we will separate the expectation in \eqref{eq:13.0} using the independent components $ G $, $ E $ and $ \ell $. Let $r\in\{0,q\}$ and $p':=[p-1]^++[p-q-1]^+$ and consider any positive 
	constants $(c_i)_{i\in\N}$. Applying~\eqref{eq:8.1} yields 
	\begin{align}
		\label{eq:sum_pqr}
		&\E\bigg[\bigg(\sum_{i=1}^n \ell_i^{1/\alpha}[G_i]^+c_i\bigg)^{p-q}
		\bigg(\sum_{j=1}^n \ell_j^{1/\alpha}[G_j]^-c_j\bigg)^{p}
		(\ell_n^{1/\alpha}[G_n]^+c_n)^r
		\bigg|\F_{-E}
		\bigg]\\
		\nonumber
		&\qquad
		\le n^{p'}\E\bigg[\sum_{i=1}^n\sum_{j=1}^n 
		\ell_i^{\frac{p-q}{\alpha}}\ell_n^{\frac{r}{\alpha}}
		\ell_j^{\frac{p}{\alpha}}
		([G_i]^+c_i)^{p-q}([G_j]^-c_j)^{p}([G_n]^+c_n)^r
		\bigg|\F_{-E}
		\bigg]\\
		\nonumber
		&\qquad
		\leq 2n^{p'}\sum_{j=1}^{n-1}\sum_{i=j+1}^n 
		\E\Big[\ell_i^{\frac{p-q}{\alpha}}\ell_n^{\frac{r}{\alpha}}
		\ell_j^{\frac{p}{\alpha}}\Big]
		\E\big[([G_i]^+)^{p-q}([G_n]^+)^r\big]
		\E\big[([G_j]^-)^{p}\big]
		c_i^{p-q}c_j^{p}c_n^r,
	\end{align} 
Note that 
the cases $ j\in\{i,n\} $ do not appear because $[x]^+[x]^-=0$. The above expression is a linear combination of monomials in $c_i$, 
	$c_n$ and $c_j$. We will analyze and bound the coefficients. 
	
	The last two expectations within the sum on the right side of the above 
	inequality can be computed using~\eqref{eq:Mellin_G} and the value of 
	their product only depends on whether $i=n$ or not. In fact, for $r\in\{0,q\}$
	\[
	\E\big[([G_i]^+)^{p-q}([G_n]^+)^r\big]
	\le 
	\max\{\E[([G_1]^+)^p],\E[([G_1]^+)^{p-q}]\E[([G_1]^+)^{q}],\E[([G_1]^+)^{p-q}]\},
	\] 
	which can be bounded by an explicit constant using~\eqref{eq:Mellin_G}. 
	
	II) Now, we obtain an important part of the bound in \eqref{eq:13.1} which is due to the stick breaking process. That is, an application of Lemma~\ref{lem:ell-mom}(b) yields the existence of 
	some $c'>0$ independent of $j$, $i$ and $n$ such that  for 
	$\theta=\tfrac{\alpha+p+r}{\alpha+2p+r}<1$, we have
	\begin{align*}
	\E[\ell_i^{(p-q)/\alpha}\ell_j^{p/\alpha}\ell_n^{r/\alpha}]
	\le c'\theta^{i+j}(1+r/\alpha)^{-n}.
	\end{align*}

	III) Now, we estimate the moments of the remaining random variables $ E_i $ which appear in the coefficients $ c_i $. By the previous steps and~\eqref{eq:sum_pqr}, we deduce that for 
	some constant $c''>0$ independent of $j$, $i$ and $n$, we have 
	\begin{align*}
		\E[A]
		&\le c'' n^{p'}\big(1+\tfrac{r}{\alpha}\big)^{-n}
		\sum_{j=1}^{n-1}\sum_{i=j+1}^n
		\theta^{i+j}\E\bigg[
		E_i^{(1-1/\alpha)(p-q)}
		E_j^{(1-1/\alpha)p}
		E_n^{(1-1/\alpha)r}Z_m^s\bigg].
	\end{align*}
	
	Next, we will show that the expectation on the right side in the above inequality 
	is bounded by a multiple of $ m^s $. As the term $\theta^{i+j}$ vanishes 
	geometrically fast, we would then obtain 
	\begin{align}
		\label{eq:15.1}
		\E[A]
		&\le c''' n^{p'}m^s\big(1+\tfrac{r}{\alpha}\big)^{-n}.
	\end{align}
	
	To prove~\eqref{eq:15.1}, observe that since $ Z_n $ in \eqref{def:Z} is a Gamma distributed 
	random variable then
	\[
	\E\big[Z_n^s\big]
	= \int_0^\infty\frac{x^s x^{n+1}}{(n+1)!}e^{-x}dx
	=\frac{\Gamma(n+s+2)}{(n+1)!}.
	\]
	Using the two-sided bounds in Stirling's formula we see that this expression 
	is bounded by a multiple of $m^s$. In fact, a similar upper bound holds for 
	$\E[Z_m^s E_i^{r_1}E_j^{r_2}E_n^{r_3}]$ with $r_1=(1-1/\alpha)(p-q)$, $r_2=(1-1/\alpha)p$ and $r_3=(1-1/\alpha)r$. Note that $r_1,r_2,r_3>-1$ in the case $ i<n $ and 
	$r_1+r_3>-1$, $ r_2>-1 $ in the case that $ i=n $. Furthermore, even in the case $ \alpha\in (0,1) $, 
	our hypotheses on $p$ and $q$ ensure that $r_1$, $r_2$ and $r_3$ satisfy these conditions. Indeed, for instance, when the 
	indices $n,i,j$ are different and $n\ge 4$, we can decompose $Z_m$ into 
	4 terms according to the index of $E$ within $Z_m$ which may equal one 
	of the indices $n,i,j$ so that, by~\eqref{eq:8.1}, 
	\begin{equation}
		\label{eq:Zn-s-growth}
		\begin{split}
			\E\big[Z_n^s E_i^{r_1}E_n^{r_2}E_j^{r_3}\big]
			&= 4^{[s-1]^+}\Big(\E[E_i^{s+r_1}]\E[E_n^{r_2}]\E[E_j^{r_3}]
			+\E[E_i^{r_1}]\E[E_n^{s+r_2}]\E[E_j^{r_3}]\\
			&\qquad+\E[E_i^{r_1}]\E[E_n^{r_2}]\E[E_j^{s+r_3}]
			+\E[E_i^{r_1}]\E[E_n^{r_2}]\E[E_j^{r_3}]\E[Z_{m-3}^s]\Big).
		\end{split}
	\end{equation}
	The quantity in~\eqref{eq:Zn-s-growth} grows as a constant multiple of 
	$m^s$ (through the $s$-moment of $Z_{m-3}$).
		Thus, we can deduce that~\eqref{eq:15.1}.

	Finally, to bound other terms in~\eqref{eq:12.1}, it is just a repetition of the above arguments but slightly easier because:
	 \begin{enumerate}[leftmargin=.75cm]
		\item the variables $\eta_+$ and $\eta_-$ are independent of the sequence 
		$(\ell_i,S_i)_{i\in\N}$.
		\item Hence, when taking expectations, the variables $\eta_+$ and $\eta_-$ 
		will factorise by independence. These variables are multiplied by powers of 
		$a_{n}=\kappa^n$ and satisfy $\E[\eta_\pm^r]=\Gamma(1+r)$ 
		for $r>-1$ so their estimation is easier. 
		\item The final bound also uses the inequality 
		$a_n\le a_{n-1}$, a consequence of Assumption~(\nameref{asm:seq_decay}). 
	\end{enumerate} 
Putting the above arguments together completes the proof of 
Lemma~\ref{lem:bound_X+X-Z}, since all eight terms decay as fast as 
$a_n^{q}$ or $(1+q/\alpha)^{-n}$ and $ n\leq m $. 
	\end{proof}
	
	\subsection{Negative and mixed moments: Proofs of Lemmas \ref{lem:inv-mom} and \ref{lem:inv_bound_X+X-Z} }
	The ideas of the proof of Lemma \ref{lem:bound_X+X-Z} can be used again for negative moments but an additional idea is required in order to use similar techniques. 
	This is provided by the following equality:
	\begin{align}
		\label{eq:inv}
	\lambda^{-p}=\Gamma(p)^{-1}\int_0^\infty x^{p-1}e^{-\lambda x}dx,\text{ for }p>0, 
	\end{align}
	which expresses the negative power $\lambda^{-p}$ using an 
	exponential expression which in our case leads to the study of the Laplace transform of the respective random variable whose inverse moment we want to estimate. The use of this technique leads to the following results.
	\begin{lem}
		\label{lem:inv-mom}
		Recall that we always assume that $\alpha\ne 1$. Below, we assume that $r\ge 0$ and $u,v,w\ge 0$ be arbitrary. \\
		{\normalfont(a)} 
		Fix any $p\in(0,\alpha\rho)$, $q\in[0,\alpha(1-\rho))$. 
		There exists a positive constant $C$ such that for any 
		$j,n\in\N$ and $T>0$, the following bound holds: 
		\begin{align}
			\label{eq:neg_mom_(a)}
			\E\bigg[\frac{\ell_{n+1}^r E_j^u \eta_+^v \eta_-^w}{X_{+,n}^pX_{-,n}^q}\bigg]
			&\le CT^{r-\frac{p+q}{\alpha}}(1+r)^{-n}.
		\end{align}
		{\normalfont(b)}
		Fix any $p\in(0,\alpha\rho)$, $q\in(0,\alpha)$.
		There exists a positive constant $ C $ such that for any 
		$j,n\in\N$ and $T>0$, the following bound holds: 
		\begin{align*}
			\E\bigg[\frac{X_{-,n}^q\ell_{n+1}^r E_j^u \eta_+^v \eta_-^w}{X_{+,n}^p}\bigg]
			&\le CT^{r+\frac{q-p}{\alpha}}(1+r)^{-n}.
		\end{align*}
	\end{lem}
	
	\begin{lem}
		\label{lem:inv_bound_X+X-Z}
		Let $p,q,r,s\ge 0$ satisfy $p\in[0,\alpha\rho)$, $q\in[0,\alpha(1-\rho))$ and 
		$r\in[0,\alpha)$. There exists a constant $C>0$ such that for any $m\ge n$ 
		and $T>0$ we have 
		\[
		\E\bigg[\frac{|\Delta_{\pm,n+1}|^r Z_{m}^s}
		{X_{+,n}^{p}X_{-,n}^q}\bigg]
		\le CT^{(2r-p-q)/\alpha}
		\Big(\big(1+\tfrac{r}{\alpha}\big)^{-n}
		+\kappa^{nr}\Big) m^{s'},
		\]
		where $s'=\1_{\{s>0\}}(s\vee 1)$. Similarly, for any $p\in[0,\alpha)$, 
		$r\in[0,p]$ and $q\in[0,\alpha(1-\rho))$, there is some $C>0$ such that 
		for all $m\ge n$ and 
		$T>0$ 
		\begin{align*}
			\E\bigg[\frac{|\Delta_{+,n+1}|^r X_{+,n+1}^{p-r} Z_{m}^s}
			{X_{-,n}^{q}}\bigg]
			\le CT^{(p-q)/\alpha}
			\Big(\big(1+\tfrac{r}{\alpha}\big)^{-n}
			+\kappa^{nr}\Big) m^{s'}.
		\end{align*}
	\end{lem}

	\begin{proof}[Proof of Lemma~\ref{lem:inv-mom}]
			Recall that we assume without loss of generality that $T=1$. The identity \eqref{eq:inv} will be used with $\lambda=X_{+,n}$ 
		(and later with $\lambda=X_{-,n}$). The resulting 
		expression will be bounded by separately integrating the variables 
		$G_1,\ldots,G_n$ and $\eta_+,\eta_-$, then and $E_1,\ldots,E_n$ and finally 
		$\ell_1,\ldots, \ell_{n+1}$ as in the proof of Lemma \ref{lem:bound_X+X-Z}. 
		These bounds require preliminary calculations for the expressions arising 
		in the inequalities developed below, so we begin with those. 
		Let $\zeta=1-1/\alpha$, $c$, $\delta$ and $\gamma$ be as defined in 
		Lemma~\ref{lem:simple-estim} of Appendix \ref{sec:appb}.  These are constants that appear in the bounds for the Laplace transform of random variables appearing in the Chambers-Mallows-Stuck representation of stable laws.
		
		\emph{Proof of {\normalfont(a)}, part 1: The case $ q=0 $. } 
		We first provide an explicit upper bound for 
		$\E\big[\ell_{n +1}^r E_j^u \eta_+^v X_{+,n}^{-p}\big]$ 
		(the case $q=0$ and $w=0$ in Lemma~\ref{lem:inv-mom}(a)) 
		using the following constants and functions. 
		
		Define $b_{\rho} := 1/(\gamma\alpha\rho) \ge 1$ and
		\begin{align}
			\begin{split}
			d_s:=&2^s\max\{1,s^s e^{-s},\Gamma(s+1)\},
			\qquad\text{ for } s\geq 0,\\
			P(x,p,q):=&\frac{((x\wedge 1)^{-p}-1)}{p}+\frac{x^{-q}(x\wedge 1)^{q-p}}{q-p},
			\quad\text{ for } x,p,q>0, q>p,\\
			Q_p(r,u)
			:=&\frac{\alpha u(1+r)-up}{p(1-p)(\alpha u(1+r)-p)(1-p/\alpha)},
			\quad\text{ for } u\in(0,1],\ p\in(0,\min\{\alpha u,1\}),\ r\ge 0.\\
			d'_u:=&\max\{\Gamma(1+u),\Gamma(1+u)\Gamma(1/\alpha),
			\Gamma(u+1/\alpha)\}=\max\{\E[E_j^u],\E[E_j^u]\E[E_k^{-\zeta}],\E[E_j^{u-\zeta}]\}.
			\end{split}\label{def:dup}
		\end{align}
		Using these definitions, it is enough to prove that for $p\in(0,\alpha\rho)$, $r,u,v\ge 0$ 
		and $j,n\in\N$ it holds
		\begin{align}
			\label{eq:mba}
		&\E\bigg[\frac{\ell_{n+1}^r E_j^u \eta_+^v}{X_{+,n}^p}\bigg]
		\le\frac{T^{r-\frac{p}{\alpha}}b_\rho c d'_u d_v}
		{\Gamma(p)(1+r)^n}\bigg(\frac{1}{p} 
		+Q_p(r,\rho)
		+(1-\rho)^nP\big(T^{-\frac{1}{\alpha}}a_n,p,\delta\big)\bigg).
		\end{align}
		The special case of~\eqref{eq:neg_mom_(a)} 
		with $q=0$, $ w> 0 $ follows from~\eqref{eq:mba}, the independence 
		$\eta_-\indep (\eta_+,E_j,\ell_{n+1},X_{+,n})$, \eqref{cond:rho} and 
		Assumption~(\nameref{asm:seq_decay}). In fact, all the above terms within the parentheses are readily bounded by a constant except $ (1-\rho)^n\kappa^{-pn} $ which is bounded due to Assumption~(\nameref{asm:seq_decay}).
		
		For the proof of~\eqref{eq:mba}, recall that 
		$X_{+,n}=\sum_{i=1}^n\ell_i^{1/\alpha}E_i^{\zeta}[G_i]^+
		+a_n \eta_+^{\zeta}$ with $\zeta=1-1/\alpha$. 
		
		Fix $p\in(0,\alpha\rho)$. A change of variables applied to the definition of 
		the Gamma function gives 
		\begin{equation}
			\label{eq:inv-mom-E_l}
			\Gamma(p)X_{+,n}^{-p}
			=\int_0^\infty x^{p-1}e^{-xX_{+,n}}dx
			\le 1/p + J_{+,p},\quad  
			\text{where}\quad
			J_{\pm,p}:= \int_1^\infty x^{p-1}e^{-xX_{\pm,n}}dx.
		\end{equation}
		Next, we bound the conditional expectation 
		$\E\big[\eta_+^vJ_{+,p}\big|\mathcal{G}\big]$, 
		where $\mathcal{G}:=\sigma(\ell_k, E_k;k\in\N)$. 
		By~\eqref{eq:exp_Er} and~\eqref{eq:exp_G} 
		(with parameter $x\ell_k^{1/\alpha}E_k^{\zeta}$), 
		that conditional expectation is smaller than 
		\begin{align}
			\notag
			\E\big[\eta_+^vJ_{+,p}\big|\mathcal{G}\big]
			=& \int_1^\infty x^{p-1}\E\big[\eta_+^ve^{-xa_n\eta_+^\zeta}\big|\mathcal{G}\big]\prod_{k=1}^n\E\big[e^{-x\ell_i^{1/\alpha}E_i^{\zeta}[G_i]^+}\big|\mathcal{G}\big]dx\\
			\le&c d_v\int_1^\infty 
			x^{p-1}\min\{1,(a_nx)^{-\delta}\}
			\prod_{k=1}^n\bigg(1-\rho + \rho\min\bigg\{1,
			\frac{b_\rho E_k^{-\zeta}}{\ell_k^{1/\alpha}x}
			\bigg\}\bigg)dx.	\label{eq:inv_mom-E_l-2}
		\end{align}

		Using that $ \ell_{n+1}=(1-U_{n+1})U_n...U_{k+1}L_k $ and $ L_{k}\leq L_{k-1} $ (see Subsection~\ref{sec:3.1}), then for any measurable function $g\ge 0$ and $k\le n$,  we have
		\begin{equation}
			\label{eq:L_k-ell_k}
			\E[\ell_{n+1}^r g(\ell_k)]
			=(1+r)^{k-n-1}\E[L_k^r g(\ell_k)]
			\le (1+r)^{k-n-1}\E[L_{k-1}^r g(\ell_k)].
		\end{equation} 
		Moreover, we have $\E[E_j^u\min\{1,E_k^{-\zeta}y\}]\le d'_u\min\{1,y\}$ by definition \eqref{def:dup}.
		In fact, if $y>1$, then $d_u'\ge \E[E_j^u]$ and if $y<1$, then $d'_uy\ge\E[E_j^uE_k^{-\zeta}y]$.
		
		Since the factors in the product of~\eqref{eq:inv_mom-E_l-2} are in $[0,1]$, 
		the inequality in~\eqref{eq:19.2}, \eqref{eq:L_k-ell_k} and $ b_\rho\geq 1 $ yields 
		\begin{align*}
			&\E\bigg[\ell_{n+1}^rE_j^u\prod_{k=1}^n
			\big(1-\rho + \rho\min\big\{1,b_{\rho}E_k^{-\zeta}
			\ell_k^{-1/\alpha}/x\big\}\big)\bigg]\\
			&\qquad
			\le (1-\rho)^n\E[\ell_{n+1}^r]\E[E_j^u]
			+ \sum_{k=1}^{n}\rho(1-\rho)^{k-1}
			\E\big[\ell_{n+1}^rE_j^u
			\min\big\{1,b_{\rho}E_k^{-\zeta}\ell_k^{-1/\alpha}/x\big\}
			\big]\\
			&\qquad
			\le \frac{(1-\rho)^nd'_u}{(1+r)^{n+1}}
			+ \frac{d'_u}{(1+r)^n}
			\sum_{k=1}^\infty\rho(1-\rho)^{k-1}(1+r)^{k-1}
			\E\big[L_{k-1}^r\min\big\{1,
			b_{\rho}\ell_k^{-1/\alpha}/x\big\}\big]\\
			&\qquad
			\le d'_u(1+r)^{-n}\big( (1-\rho)^n
			+ b_\rho A_\rho(x)\big),
		\end{align*}
		where we define 
		$A_{\rho}(x):=\sum_{k=1}^\infty\rho(1-\rho)^{k-1}
		(1+r)^{k-1}\E[L_{k-1}^r\min\{1,\ell_k^{-1/\alpha}/x\}]$ for $x>0$.

		Hence, the inequality $\E[\ell_{n+1}^r E_j^u \eta_+^vX_{+,n}^{-p}]\Gamma(p)
		\le d'_ud_v(1+r)^{-n}(1/p+cK)$ holds for
		\[
		K:=\int_1^\infty x^{p-1}\min\big\{1,(a_n x)^{-\delta}\big\}
		((1-\rho)^n + b_\rho A_\rho(x))dx.
		\]
		
			Next, we apply Lemma~\ref{lem:simple-estim}(c) to find a formula for 
		$\int_1^\infty x^{p-1}A_{\rho}(x)dx$. Note that $p<\alpha\rho$ implies 
		$\tfrac{(1-\rho)(1+r)}{1+r-p/\alpha}=\tfrac{1+r-\rho(1+r)}{1+r-p/\alpha}<1$, 
		so Fubini's theorem and Lemmas~\ref{lem:simple-estim}(c) 
		and~\ref{lem:ell-mom}(c) yield 
		\begin{align}
			\int_1^\infty x^{p-1}A_\rho(x)dx
			\nonumber
			&=\sum_{k=1}^\infty\rho(1-\rho)^{k-1}
			(1+r)^{k-1}\E\bigg[
			L_{k-1}^r\int_1^\infty x^{p-1}\min\{1,(\ell_k^{1/\alpha}x)^{-1}\}dx\bigg]\\
			\nonumber
			&=\sum_{k=1}^\infty\rho(1-\rho)^{k-1}
			(1+r)^{k-1}\E\bigg[L_{k-1}^r
			\bigg(\frac{\ell_k^{-p/\alpha}}{p(1-p)}
			-\frac{1}{p}\bigg)\bigg]\\
			&=\sum_{k=1}^\infty\rho(1-\rho)^{k-1}
			(1+r)^{k-1}\bigg(\frac{(1+r-p/\alpha)^{1-k}}{p(1-p)(1-p/\alpha)}
			-\frac{(1+r)^{1-k}}{p}\bigg)
			=Q_p(r,\rho).
			\label{eq:int_A}
		\end{align}

		Thus~\eqref{eq:mba} follows from~\eqref{eq:int_A} and 
		Lemma~\ref{lem:simple-estim}(c) since for any $p<\alpha\rho$ and 
		$n\in\N$ we have 
		\[
		K\le\int_1^\infty x^{p-1}\big(
		(1-\rho)^n\min\{1,(a_nx)^{-\delta}\} + b_\rho A_\rho(x)\big)dx
		\le b_\rho\big[
		(1-\rho)^nP\big(a_n,p,\delta\big) + Q_p(r,\rho,p)\big].
		\]
		
		\emph{Proof of {\normalfont(a)}, part 2. The case $ q\in (0, \alpha(1-\rho))$.} 
		The general case of~\eqref{eq:neg_mom_(a)} for $q>0$ follows similarly but with lengthier expressions.
		Recall that $B(\cdot,\cdot)$ denotes the beta function and define for any 
		$u\in(0,1]$, $p\in(0,\alpha u\wedge 1)$, $q\in(0,\alpha\wedge1)$ and 
		$r\ge 0$, 
		\[
		R_{p,q}(r,u)
		:=\frac{(\Gamma(1/\alpha)\vee 1) B\big(1+r-p/\alpha,1-q/\alpha)
			 u(1-u)(1+r)^2(1+r-p/\alpha)}
		{pq(1-p)(1-q)(1-p/\alpha)(u(1+r)-p/\alpha)}.
		\]
		Fix $p\in(0,\alpha\rho)$, $q\in(0,\alpha(1-\rho))$ and $r,u,v,w\ge 0$. 
		We will prove that for all $j,n\in\N$, we have  
		\begin{align*}
			\E\bigg[\frac{\ell_{n+1}^r E_j^u \eta_+^v \eta_-^w}{X_{+,n}^pX_{-,n}^q}\bigg]
			&\le \frac{T^{r-\frac{p+q}{\alpha}}b_\rho b_{1-\rho}c^2 d'_u d_v d_w }
			{\Gamma(p)\Gamma(q)(1+r)^n}\Big[
			(1-\rho)^n P\big(T^{-\frac{1}{\alpha}}a_n,p,\delta\big)/q
			+\rho^nP\big(T^{-\frac{1}{\alpha}}a_n,q,\delta\big)/p\\
			&\qquad
			+\big((1-\rho)^n+\rho^n\big)
			P\big(T^{-\frac{1}{\alpha}}a_n,p,\delta\big)
			P\big(T^{-\frac{1}{\alpha}}a_n,q,\delta\big)\\
			&\qquad
			+1/(pq)
			+Q_p(r,\rho)/q
			+Q_q(r,1-\rho)/p
			+ R_{p,q}(r,\rho) + R_{q,p}(r,1-\rho)
			\Big],
		\end{align*}
		Once this bound is proven the final result follows as above, 
		by~\eqref{cond:rho} and Assumption~(\nameref{asm:seq_decay}). 
		Indeed, 
		$((1-\rho)^n+\rho^n)P(T^{-1/\alpha}a_n,p,\delta)P(T^{-1/\alpha}a_n,q,\delta)
		\le(1/p+1/(\delta-p))(1/q+1/(\delta-q))2\kappa^{n\alpha-n(p+q)}$, 
		by Assumption~(\nameref{asm:seq_decay}), which is bounded for $n\in\N$ 
		because $p+q<\alpha$.
		
		Applying~\eqref{eq:inv-mom-E_l} twice, we obtain 
		\begin{equation}
			\label{eq:Feller-bi-trick}
			\Gamma(p)\Gamma(q)X_{+,n}^{-p}X_{-,n}^{-q}
			\le 1/(pq) + J_{-,q}/p + J_{+,p}/q + J_{+,p}J_{-,q}.
		\end{equation}
		It remains to multiply~\eqref{eq:Feller-bi-trick} by 
		$\ell_n^rE_j^u\eta_+^v\eta_+^w$ and take expectations.
		The first term in~\eqref{eq:Feller-bi-trick} yields the inequality
		$\E[\ell_{n+1}^rE_j^u\eta_+^v\eta_-^w]/(pq)\le(1+r)^{-n-1}d'_ud_vd_w/(pq)$. 
		The second and third terms are bounded as in the special case $q=0$, 
		since $\eta_+$ (resp. $\eta_-$) is independent of $X_{-,n}$ (resp. $X_{+,n}$). 
		
		It remains to bound $\E[\ell_{n+1}^rE_j^u\eta_+^v\eta_-^w J_{+,p}J_{-,q}]$. 
		Note that applying Lemma~\ref{lem:simple-estim}(b) twice gives 
		\begin{align*}
			\E[e^{-x[G_1]^+-y[G_1]^-}]\le \Upsilon(x,y)
			:=\rho\min\{1,b_{\rho}/x\}+(1-\rho)\min\{1,b_{1-\rho}/y\}\leq 1.
		\end{align*}
		Recall $\mathcal{G}=\sigma(\ell_k,E_k;k\in\N)$ and 
		apply~\eqref{eq:exp_Er} to 
		$\E[\ell_{n+1}^rE_j^u\eta_+^v\eta_-^w J_{+,p}J_{-,q}|\mathcal{G}]
		=\int_1^\infty\int_1^\infty x^{p-1}y^{q-1}S(x,y)dydx$, where
		\begin{align*}
			S(x,y)
			&:=\ell_{n+1}^rE_j^u
			\E\big[\eta_+^v \eta_-^w e^{-xX_{+,n}-yX_{-,n}}\big|\mathcal{G}\big]\\
			&=\ell_{n+1}^rE_j^u\E\Big[\eta_+^ve^{-xa_n\eta_+^{\zeta}}\Big]
			\E\Big[\eta_-^ve^{-ya_n\eta_-^{\zeta}}\Big]
			\prod_{k=1}^n\E\Big[e^{-\ell_k^{1/\alpha}E_k^\zeta\left(x [G_k]^+
				-y  [G_k]^-\right)}
			\Big|\mathcal{G}\Big]\\
			&\le c^2 d_v d_w 
			\min\{1,(a_n x)^{-\delta}\}
			\min\{1,(a_n y)^{-\delta}\}
			\ell_{n+1}^rE_j^u
			\prod_{k=1}^n \Upsilon(E_k^{\zeta}\ell_k^{1/\alpha}x,
			E_k^{\zeta}\ell_k^{1/\alpha}y).
		\end{align*}
		The inequality 
		$\E[E_j^u\min\{1,E_k^{-\zeta}x\}\min\{1,E_1^{-\zeta}y\}]
		\le d'_u(\Gamma(1/\alpha)\vee 1)\min\{1,x\}\min\{1,y\}$ for 
		$k\ge 2$, along with \eqref{eq:19.2a} and~\eqref{eq:L_k-ell_k} yield 
		\begin{align*}
			&\E\bigg[\ell_{n+1}^rE_j^u
			\prod_{k=1}^n\Upsilon(E_k^{\zeta}\ell_k^{1/\alpha}x,
			E_k^{\zeta}\ell_k^{1/\alpha}y)\bigg]\\
			&\hspace{4mm}
			\le ((1-\rho)^n +\rho^n)\E[\ell_{n+1}^r]\E[E_j^u]
			+ \sum_{k=2}^{n}\rho(1-\rho)^{k-1}
			\E\bigg[\ell_{n+1}^r E_j^u
			\min\bigg\{1,
			\frac{b_{\rho}E_k^{-\zeta}}{\ell_k^{1/\alpha}x}\bigg\}
			\min\bigg\{1,
			\frac{b_{1-\rho}E_1^{-\zeta}}{\ell_1^{1/\alpha}y}\bigg\}
			\bigg]\\
			&\hspace{12mm}
			+ \sum_{k=2}^{n}(1-\rho)\rho^{k-1}
			\E\bigg[\ell_{n+1}^r E_j^u
			\min\bigg\{1,
			\frac{b_{\rho}E_1^{-\zeta}}{\ell_1^{1/\alpha}x}\bigg\}
			\min\bigg\{1,
			\frac{b_{1-\rho}E_k^{-\zeta}}{\ell_k^{1/\alpha}y}\bigg\}
			\bigg]\\
			&\hspace{4mm}
			\le b_\rho b_{1-\rho}d'_u (1+r)^{-n}
			[(1-\rho)^n+\rho^n + (\Gamma(1/\alpha)\vee 1) 
			(B_{\rho}(x,y) + B_{1-\rho}(y,x))],
		\end{align*}
		where 
		$B_s(x,y) :=\sum_{k=2}^\infty s(1-s)^{k-1}(1+r)^{k}
		\E\big[L_{k-1}^r
		\min\big\{1,\ell_k^{-1/\alpha}/x\big\}
		\min\big\{1,\ell_1^{-1/\alpha}/y\big\}\big]$ for $x,y>0$.
		
		Next we give a simple bound on some integrals of $B_s$. 
		Recall that we have $p+q<\alpha$ and $\E[U^r(1-U)^s]=B(r+1,s+1)$. 
		Thus 
		an application of Fubini's theorem, \eqref{eq:tail_int_min} and 
		Lemma~\ref{lem:ell-mom}(c) yields 
		\begin{align*}
			&\int_1^\infty\int_1^\infty x^{p-1}y^{q-1}B_\rho(x,y)dydx\\
			&\qquad
			=\sum_{k=2}^\infty\frac{\rho(1-\rho)^{k-1}}{(1+r)^{-k}}
			\E\bigg[L_{k-1}^r
			\int_1^\infty\int_1^\infty x^{p-1}y^{q-1}
			\min\big\{1,\ell_k^{-1/\alpha}/x\big\}
			\min\big\{1,\ell_1^{-1/\alpha}/y\big\}dydx
			\bigg]\\
			&\qquad
			=\sum_{k=2}^\infty\frac{\rho(1-\rho)^{k-1}}{(1+r)^{-k}}
			\E\bigg[L_{k-1}^r
			\bigg(\frac{\ell_k^{-p/\alpha}}{p(1-p)}-\frac{1}{p}\bigg)
			\bigg(\frac{\ell_1^{-q/\alpha}}{q(1-q)}-\frac{1}{q}\bigg)
			\bigg]\\
			&\qquad
			\le\sum_{k=2}^\infty\frac{\rho(1-\rho)^{k-1}}{(1+r)^{-k}}
			\E\bigg[
			\frac{L_{k-1}^r\ell_k^{-p/\alpha}\ell_1^{-q/\alpha}}{pq(1-p)(1-q)}
			\bigg]
			=R_{p,q}(r,\rho)/(\Gamma(1/\alpha)\vee 1).
		\end{align*}
		Putting all the above arguments together, the following inequalities imply 
		part (a): 
		\begin{align*}
			&\E[\ell_{n+1}^rE_j^u\eta_+^v\eta_-^w J_{+,p}J_{-,q}]\\
			&\hspace{16mm}
			\le \frac{b_{\rho}b_{1-\rho}c^2 d'_u d_v d_w}{(1+r)^n}
			\int_1^\infty\int_1^\infty 
			x^{p-1}y^{q-1}
			\min\{1,(a_n x)^{-\delta}\}
			\min\{1,(a_n y)^{-\delta}\}\\
			&\hspace{24mm}\times
			\big((1-\rho)^n+\rho^n
			+(\Gamma(1/\alpha)\vee 1) \left(B_{\rho}(x,y)+B_{1-\rho}(y,x)\right)
			\big)dydx\\
			&\hspace{16mm}
			\le \frac{b_{\rho}b_{1-\rho}c^2 d'_u d_v d_w}{(1+r)^n}\big[
			\big((1-\rho)^n+\rho^n\big)
			P\big(a_n,p,\delta\big)
			P\big(a_n,q,\delta\big)
			+R_{p,q}(r,\rho) + R_{q,p}(r,1-\rho)\big].
		\end{align*}
		
		\emph{Proof of {\normalfont(b)}.} 
		Again, we use a slightly different combination of 
		some of the previously explained ideas. We begin using \eqref{eq:8.1} to obtain
		\begin{equation}
			\label{eq:mix-mom}
			\frac{\Gamma(p)X_{-,n}^q}{2^{(q-1)^+}X_{+,n}^p}
			\le \frac{\Gamma(p)a_n^q\eta_-^{q\zeta}}{X_{+,n}^p}
			+ (X_{-,n} -a_n\eta_-^{\zeta})^q
			\bigg(\frac{1}{p} + \int_1^\infty x^{p-1}e^{-xX_{+,n}}dx\bigg).
		\end{equation}
		It remains to multiply the above expression by $\ell_{n+1}^r E_j^u\eta_+^v\eta_-^w$ and take expectations. 
		
		The first term in~\eqref{eq:mix-mom} can be bounded as in part (a). The 
		second term $(X_{-,n} -a_n\eta_-^{1-1/\alpha})^q/p$ in~\eqref{eq:mix-mom} 
		can be handled as in Lemma~\ref{lem:bound_X+X-Z} (see~\eqref{eq:12.1} 
		and~\eqref{eq:15.1}). Indeed, we have 
		\begin{equation}
			\label{eq:1.51}
			\big(X_{-,n}-a_n\eta_-^{\zeta}\big)^q\ell_{n+1}^r E_j^u\eta_+^v\eta_-^w
			= \bigg(
			\sum_{k=1}^n \ell_k^{\frac{1}{\alpha}}E_k^{\zeta}[G_k]^-
			\ell_{n+1}^{\frac{r}{q}} E_j^{\frac{u}{q}}
			\bigg)^q\eta_+^v\eta_-^w.
		\end{equation}
		
		The expected value of~\eqref{eq:1.51} may be bounded via $L^q$-seminorms: 
		denote by $\|\vartheta\|_q=\E[\vartheta^q]^{1/q'}$ the $L^q$-seminorm of 
		$\vartheta$ where $q'=q\vee 1$ (which is a true norm if $q\ge 1$). 
		Let $g_q=\E[([G_k]^-)^q]$ and $h_u=\max\{\Gamma(1+u+q\zeta),
		\Gamma(1+q\zeta)\Gamma(1+u)\}$; observe that when $\alpha<1$, 
		we have $q\zeta>\alpha-1>-1$. Then the triangle inequality and the 
		independence gives 
		\begin{align*}
			\bigg\|
			\sum_{k=1}^n \ell_k^{\frac{1}{\alpha}}E_k^{\zeta}[G_k]^-
			\ell_{n+1}^{\frac{r}{q}} E_j^{\frac{u}{q}}
			\bigg\|_q^{q'}
			&\le \bigg(
			\sum_{k=1}^n \Big\|\ell_k^{\frac{1}{\alpha}}\ell_{n+1}^{\frac{r}{q}}\Big\|_q
			\Big\|E_k^{\zeta}E_j^{\frac{u}{q}}\Big\|_q
			\big\|[G_k]^-\big\|_q
			\bigg)^{q'}\\
			&\le\frac{h_u g_q B(1+\tfrac{q}{\alpha},1+r)}{(1+r)^n}
			\bigg(\sum_{k=1}^n \Big(\frac{1+r+q/\alpha}{1+r}\Big)^{(1-k)/q'}\bigg)^{q'}\\
			&\le\frac{h_u g_q B(1+\tfrac{q}{\alpha},1+r)}{(1+r)^n}
			\frac{1+r+q/\alpha}{\big((1+r+q/\alpha)^{1/q'}-(1+r)^{1/q'}\big)^{q'}},
		\end{align*}
		which completes the bound on the second term in~\eqref{eq:mix-mom} 
		once one notes that $\eta_+$ and $\eta_-$ are independent from the other 
		variables and $\E[\eta_+^v\eta_-^w]=\Gamma(v+1)\Gamma(w+1)$. 
		
		The third term in~\eqref{eq:mix-mom} may be bounded as follows. 
		Set $s=q/2<\alpha/2\le 1$, then we may use~\eqref{eq:8.1} to obtain 
		\begin{align*}
			&\E\bigg[\bigg(\sum_{k=1}^n 
			\ell_k^{\frac{1}{\alpha}}E_k^{\zeta}[G_k]^-
			\bigg)^q\ell_{n+1}^r E_j^u\eta_+^v\eta_-^w
			\int_1^\infty x^{p-1} e^{-xX_{+,n}}dx\bigg]\\
			&\qquad
			\le\int_1^\infty x^{p-1}\E\bigg[\bigg(\sum_{k=1}^n 
			\ell_k^{\frac{s}{\alpha}}E_k^{s\zeta}([G_k]^-)^s\bigg)^2
			\ell_{n+1}^r E_j^ue^{-xX_{+,n}}\eta_+^v\eta_-^w\bigg]dx\\
			&\qquad
			= 2\int_1^\infty x^{p-1}\sum_{k=1}^n\sum_{i=k+1}^n
			\E\Big[ 
			\ell_k^{\frac{s}{\alpha}}E_k^{s\zeta}([G_k]^-)^s
			\ell_i^{\frac{s}{\alpha}}E_i^{s\zeta}([G_i]^-)^s
			\ell_{n+1}^r E_j^ue^{-xX_{+,n}}\eta_+^v\eta_-^w\Big]dx\\
			&\qquad\qquad
			+ \int_1^\infty x^{p-1}\sum_{k=1}^n 
			\E\Big[
			\ell_k^{\frac{q}{\alpha}}E_k^{q\zeta}([G_k]^-)^q
			\ell_{n+1}^r E_j^ue^{-xX_{+,n}}\eta_+^v\eta_-^w\Big]dx.
		\end{align*}
		The previous expression can be dealt with as in (a) and (b). That is, first we 
	 average with respect to $(G_n)_{n\in\N}$, using that 
		$\E[([G_k]^-)^se^{-x[G_k]^+}]=\E[([G_k]^-)^s]$. In particular, one uses Lemma \ref{lem:simple-estim} (b) for the terms containing exponentials of $G$ and \eqref{eq:Mellin_G} for the terms which do not contain  exponentials of $G$ in order to obtain a similar estimate as in \eqref{eq:inv_mom-E_l-2}. For the terms which contain $ \eta_{\pm} $, one uses Lemma \ref{lem:simple-estim} (a). Next, one takes expectations with respect to $ (E_n)_{n\in\N} $. As in the proof of Lemma \ref{lem:inv-mom} (a), one defines the appropriate $ d_u' $ which will bound all the required powers of $ E $. 
		Finally, as in steps II) 
		and III) of the proof of Lemma \ref{lem:bound_X+X-Z}, we take the expectations for $ (\ell_n)_{n\in\N}  $ using Lemma \ref{lem:ell-mom} (b). Each term in the first 
		sum can be bounded by $C'\theta_1^{i+k}(1+r)^{-n}$ for some $C'>0$, 
		$\theta_1\in(0,1)$ (independent of $i,k,n$) and all $k<i\le n$, 
		whereas each term in the second sum can be bounded by some 
		$C''\theta_2^k(1+r)^{-n}(1+(1-\rho)^nP(a_n,p,\delta))$ for some $C''>0$, 
		$\theta_2\in(0,1)$ (independent of $i,k,n$) and all $k\le n$. The claim of 
		part (b) then follows, completing the proof. 
	\end{proof}
	
	\begin{proof}[Proof of Lemma~\ref{lem:inv_bound_X+X-Z}]
		We will prove the case $s>0$, as the case $ s=0 $ is very similar. 
		The result is a consequence of Lemma~\ref{lem:inv-mom}(a). 
		Since $[S_k]^+[S_k]^-=0$, observe that using \eqref{eq:8.1} 
		\begin{align*}
			Z_{m}^s
			&\le (m+2)^{[s-1]^+}\bigg(\eta_+^s + \eta_-^s + \sum_{k=1}^{m}E_k^s\bigg)
			\quad\text{and}\quad
			\Delta^+_{n+1}
			=\ell_{n+1}^{1/\alpha}[S_{n+1}]^+ 
			+ (a_{n+1}-a_n)\eta_+^{\zeta}.
		\end{align*}
		
		Recall that $S_{n+1}=E_{n+1}^{\zeta}G_{n+1}$, where $G_{n+1}$ and 
		$E_{n+1}$ are independent of each other and of every other random variable 
		in the expectations of the statement. Similarly, $(E_{n+2},\ldots,E_m)$ is 
		independent of every other random variable in the expectations of the 
		statement. An application of Lemma~\ref{lem:inv-mom}(a) 
		(and~\eqref{eq:Mellin_G}) gives the claim if one uses hypothesis 
		(\nameref{asm:seq_decay}). The second claim follows similarly using 
		Lemma~\ref{lem:inv-mom}(b). In particular, note that the restriction on $r$ in the case (a) is due to the $r$-th moment of $G_{n+1}$ while in the case (b), the restriction on $r$ ensures 
		that the power $p-r$ of $X_{+,n+1}$ is non-negative.
	\end{proof}
	\begin{rem}Note that in the above results the parameters for the negative moments can not achieve their upper limit. This is the main reason for not being able to achieve $ \alpha'=\alpha $ in Theorem \ref{thm:density_bound}.
		\end{rem}

	\section{Final remarks}
	\label{sec:conclusions}
	In this section, we gathered some extra technical comments that may be useful for other developments. \\
{\normalfont(i)} 	Our claim for nearly-optimal bound is not proven in two particular situations. That is, in the special case where the stable process is of infinite variation and has 
	only negative jumps (i.e. $\alpha\rho=1$), $\ov{X}_T$ has exponential 
	moments and therefore our our bound is suboptimal for large $y$. However, the optimality 
	of the bound is retained in a neighborhood of $0$. Although we do not provide the details here, our methods 
	could be applied to obtain the corresponding exponential bound for the 
	density as $y\to\infty$ in this special case, 
	one may use the techniques in the proof of Proposition \ref{prop:Theta_bound} (a) and (c) to obtain exponential bounds in $x_+$. In those cases, we would show that the densities and all their derivatives decay faster than any polynomial $x_+^{-p}$, $p>0$, as $x_+\to\infty$. 
In the other extreme, when the infinite variation process has only positive jumps (i.e. $\alpha(1-\rho)=1$), analogous remarks apply. \\
	{\normalfont(ii)} We stress that the constant $C$ in 
		Proposition~\ref{prop:Theta_bound} is independent of $n$ and $x_\pm>0$. 
		In fact, it can be shown that  $(\alpha-\alpha')C$ is bounded as 
		$\alpha'\to\alpha$. 

	\appendix
	
	\section{Moments of the stick-breaking process}
	\label{app:stick_breaking}
	
	Recall from Subsection~\ref{sec:3.1}
	the definition of stick-breaking process $\ell$ on $[0,T]$
	and its remainders $(L_{k-1})_{k\in\N}$.
	\begin{lem}
		\label{lem:ell-mom}
		{\normalfont(a)} 
		Let $n\in\N$ and $p_1,\ldots,p_n>-1$ satisfy $q_k:=\sum_{i=k+1}^n p_i>-1$ 
		for $k\in\{1,\ldots,n\}$ (with $q_n:=0$). Let $q_0:=\sum_{k=1}^np_k$, 
		then we have 
		\begin{equation}
			\label{eq:ell-mom}
			\E\left[\prod_{k=1}^n \ell_k^{p_k}\right]
			=T^{q_0}\prod_{k=1}^n B(1+p_k,1+q_k),
		\end{equation}
		where $B(\cdot,\cdot)$ denotes the beta function. In particular 
		$\E\left[ \ell_k^{p}\right]= T^p(1+p)^{-k}$ for $k\geq 1$ and $p>-1$.\\
		{\normalfont(b)} 
		Let $p,q,r\ge 0$ and define $\theta=\frac{1+r+p\vee q}{1+r+p+q}\le 1$. 
		Then there is some $C>0$ such that for any $1\le j\le k\le n$ and $T>0$, 
		we have $\E[\ell_j^p\ell_k^q\ell_n^r]\le CT^{p+q+r}\theta^{p+q}(1+r)^{-n}$.\\
		{\normalfont(c)}
		If $p+q,q,r>-1$ and $k\ge 2$ then $\E[L_{k-1}^p\ell_k^q\ell_1^r]
		\le T^{p+q+r}{B(1+p+q,1+r)}(1+q)^{-1}(1+p+q)^{2-k}$.
	\end{lem}
	
	\begin{proof}
		(a) Recall $\ell_k=T(1-U_k)\prod_{i=1}^{k-1}U_i$ for $k\ge 1$,  implying
		$\prod_{k=1}^n \ell_k^{p_k}=T^{q_0}\prod_{k=1}^n U_k^{q_k}(1-U_k)^{p_k}$.
		Equation~\eqref{eq:ell-mom} follows from the identity 
		$\E[U_1^p(1-U_1)^q]=B(1+p,1+q)$ and the independence of the uniformly 
		distributed random variables $U_1,\ldots,U_n$. 
		
		(b) Applying~\eqref{eq:ell-mom} yields (note that some factors in the 
		product become 1 in this case)
		\[
		\frac{\E[\ell_j^p\ell_k^q\ell_n^r]}{T^{p+q+r}}
		=(1+r)^{k-n}\times\begin{cases}
			\begin{split}
				&B(1+p,1+q+r)(1+q+r)^{j-k-1}\\
				&\qquad\times B(1+q,1+r)(1+p+q+r)^{1-j},
			\end{split}
			&j<k<n,\\
			(1+p+q+r)^{1-j}B(1+p,1+q+r)(1+q+r)^{j-k},
			&j<k= n,\\
			(1+p+q+r)^{1-j}B(1+p+q,1+r)(1+q+r)^{j-k},
			&j=k< n\\
			(1+p+q+r)^{-j}(1+q+r)^{j-k},
			&j=k= n.
		\end{cases}
		\]
		
		In order to avoid considering four different 
		cases to obtain the claimed bound, observe that $(1+b+c)/(1+a+c)\le(1+b)/(1+a)$ for $0\le a\le b$ and 
		$c\ge 0$. The claim then follows easily. 
		
		(c) The proof is analogous to that of part (a). 
	\end{proof}
	\section{Technical lemmas for moment estimates}
	\label{sec:appb}
	\begin{lem}
		\label{lem:seq_in_01}
		Let $x_1,x_2,\ldots,y_1,y_2,\ldots\in[0,1]$. Then for any $r\in[0,1]$ and 
		$n\in\N$, it holds that 
		\begin{align}
			\label{eq:19.2}
			\prod_{k=1}^n ((1-r)+r x_k)
			&\le (1-r)^n+\sum_{k=1}^n r(1-r)^{k-1}x_k\\
			\label{eq:19.2a}
			\prod_{k=1}^n((1-r)y_k+r x_k)
			&\le r^n + (1-r)^n + \sum_{k=2}^n r(1-r)^{k-1}x_k y_1
			+\sum_{k=2}^n(1-r)r^{k-1}x_1 y_k.
		\end{align}
	\end{lem}
	
	\begin{proof}
		Identity~\eqref{eq:19.2} follows by developing the product term by term 
		and using the fact that every term in the product is bounded by $1$. Indeed, 
		we have 
		\[
		\prod_{k=1}^n ((1-r)+r x_k)
		\le r x_1 + (1-r)\prod_{k=2}^n ((1-r)+r x_k)
		\le \cdots 
		\le (1-r)^n+\sum_{k=1}^n r(1-r)^{k-1}x_k.
		\]
		The same ideas yield~\eqref{eq:19.2a}:
		\begin{align*}
			\prod_{k=1}^n((1-r)y_k+r x_k)
			&=(1-r)y_1\prod_{k=2}^n((1-r)y_k+r x_k)
			+r x_1\prod_{k=2}^n((1-r)y_k+r x_k)\\
			&\le y_1\bigg((1-r)^n+\sum_{k=2}^n r(1-r)^{k-1}x_k\bigg)
			+x_1\bigg(r^n+\sum_{k=2}^n(1-r)r^{k-1}y_k\bigg)\\
			&\le r^n + (1-r)^n + \sum_{k=2}^n r(1-r)^{k-1}x_k y_1
			+\sum_{k=2}^n(1-r)r^{k-1}x_1 y_k.
			\qedhere
		\end{align*}
	\end{proof}
	In the next lemma, we give bounds for the Laplace transforms of random variables related to the Chambers-Mallows-Stuck representation of stable laws.
	\begin{lem}
		\label{lem:simple-estim}
		{\normalfont(a)}
		Suppose $\alpha\ne 1$. 
		For any $s\ge 0$ define $d_s=2^s\max\{1,s^s e^{-s},\Gamma(s+1)\}$ 
		(with the convention $0^0=1$) and let $\zeta=1-1/\alpha$. Define
		\[
		(c,\delta)
		:=\begin{cases}
			(\max\{1,\int_0^\infty \exp(-y^\zeta)dy\},1/\zeta),&\text{ if }\alpha\in (1,2),\\
			((2+1/|\zeta|)\max\{1,(2e^{-1}/\alpha)^{1/\alpha}\},1),&\text{ if }\alpha\in (0,1).
		\end{cases}
		\]
		If $Y$ is a  exponential variable with unit mean, then for $s,x\ge 0$ we have 
		\begin{equation}
			\label{eq:exp_Er}
			\E\big[Y^s e^{-xY^\zeta}\big]\le c d_s \min\{1,x^{-\delta}\}.
		\end{equation}
		{\normalfont(b)} 
		Recall that $ G=g(V) $ where $ V $ follows a $ U(-\frac \pi 2,\frac \pi 2) $ law. Suppose $\alpha\ne 1$ then for any $x>0$, it holds that 
		\begin{equation}
			\label{eq:exp_G}
			\E[e^{-xG}|G>0]\le \min\{1,(\gamma\alpha\rho x)^{-1}\},
			\enskip\text{and}\enskip
			\E[e^{-x[G]^+}]\le 1-\rho +\rho \min\{1,(\gamma\alpha\rho x)^{-1}\},
			\enskip\text{where}
		\end{equation}
		\[
		\gamma
		:=\begin{cases}
			1, &\text{ if }\alpha\in (1,2),\\
			\min\big\{\cos\big(\pi\big(\frac{1}{2}-\rho\big)\big),
			\cos\big(\pi\big(\frac{1}{2}-\alpha\rho\big)\big)\big\}^{1/\alpha-1}, 
			&\text{ if }\alpha\in (0,1).
		\end{cases}
		\]
		{\normalfont(c)}
		For any $p\in\R$, $q>p$ and $b>0$ we have 
		\begin{equation}
			\label{eq:tail_int_min}
			\int_1^\infty x^{p-1}\min\{1,(bx)^{-q}\}dx
			=P(b,p,q),\text{ where }
			P(b,p,q):=
			\frac{(b\wedge 1)^{-p}-1}{p}
			+ \frac{b^{-q}(b\wedge 1)^{q-p}}{q-p}.
		\end{equation}
		Moreover, if $q=1>b$, then the above integral equals
		$\frac{1}{p(1-p)}b^{-p}-\frac{1}{p}$. 
	\end{lem}
	
	\begin{proof}
		(a) Consider first the case $\alpha>1$. For $s\geq 0$ it holds that 
		$\E\big[Y^s e^{-xY^\zeta}\big]\le \E\big[Y^s\big]
		=\Gamma(s+1)\le d_s\le c d_s$. Moreover, since $y^se^{-y}\leq d_s$  for all $y\in\R_+$, we have
		\[
		\E\big[Y^s e^{-xY^\zeta}\big]
		=\int_0^\infty y^s e^{-xy^\zeta-y}dy
		\le d_s\int_0^\infty e^{-xy^\zeta}dy
		\le c d_s x^{-1/\zeta}
		= c d_s x^{-\delta},
		\]
		implying~\eqref{eq:exp_Er} for $\alpha>1$. 
		
		Suppose $\alpha<1$, so that $\zeta<0$. As before, we have 
		$\E\big[Y^s e^{-xY^\zeta}\big]\le\E\big[Y^s\big]=\Gamma(s+1)\le cd_s$. 
		Without loss of generality
		assume that 
		$x\ge 1$. Recall that $ y^se^{-y}\leq d_s$ for all $y\in\R_+$. Hence 
		\[
		\E\big[Y^s e^{-xY^\zeta}\big]
		=\int_0^\infty y^s e^{-xy^\zeta-y}dy
		\le d_s\int_0^\infty e^{-xy^\zeta-y/2}dy.
		\]
		Decomposing this integral into two parts yields 
		\begin{align}
			\label{eq:22.1}
			\int_0^\infty e^{-xy^\zeta-y/2}dy
			\le\int_0^{x^\alpha} e^{-xy^\zeta}dy + \int_{x^\alpha}^\infty e^{-y/2}dy
			\le  2(2e^{-1}/\alpha)^{1/\alpha}x^{-1}+\int_0^{x^\alpha} e^{-xy^\zeta}dy,  
		\end{align}
		since $(2e^{-1}/\alpha)^{1/\alpha}\geq ye^{-y^\alpha/2}$ for any $y\ge 0$. 
		For the remaining integral in~\eqref{eq:22.1}, note 
		that for any $s\ge 0$ and $z\ge 1$ we have 
		$\int_z^\infty y^{-s}e^{-y}dy\le  z^{-s} e^{-z}$ 
		and change of variables $u=xy^\zeta$ (recall that $1+\alpha \zeta=\alpha$): 
		\[
		\int_0^{x^\alpha} e^{-xy^\zeta}dy
		=\frac{x^{-1/\zeta}}{|\zeta|}\int_{x^\alpha}^\infty u^{1/\zeta-1}e^{-u}du
		\le \frac{1}{|\zeta|}x^{-(1-\alpha)/\zeta-\alpha}e^{-x^\alpha}
		= \frac{1}{|\zeta|}e^{-x^\alpha}
		\le \frac{1}{|\zeta|}(e^{-1}/\alpha)^{1/\alpha}x^{-1}.
		\]
		This concludes the proof of~\eqref{eq:exp_Er}.
		
		(b) The conditional law of $G$ given $G>0$ is that of $g(V)$, where $V$ is 
		uniformly distributed on the interval 
		$(\pi\big(\frac{1}{2}-\rho\big),\frac{\pi}{2}\big)$. Define $z$ by
		$v=z\pi\rho-\pi\big(\rho-\frac{1}{2}\big)$ and note that 
		$v\in(\pi\big(\frac{1}{2}-\rho\big),\frac{\pi}{2}\big)$ if and only if 
		$z\in(0,1)$.
		Then, for 
		$v\in(\pi\big(\frac{1}{2}-\rho\big),\frac{\pi}{2}\big)$,
		we claim that the function $g$ in~\eqref{eq:g-function} 
		satisfies
		\[
		g(v)\geq  \gamma \sin(z\pi\alpha\rho)>0.
		\]
		Indeed, in the case that $\alpha>1$, the product of cosines in the denominator of $g$ 
		is bounded above by $1$, the exponents are positive 
		and
		$\alpha\rho\in[\alpha-1,1]$ with $ \gamma=1 $. 
		Similarly, in the case that $\alpha<1$, then $v-z\pi\alpha\rho\in
		\big(\pi\big(\frac{1}{2}-\rho\big),\pi\big(\frac{1}{2}-\alpha\rho\big)\big)$. 
		Since $\rho,\alpha\rho\in(0,1)$, we have 
		$\cos^{1/\alpha-1}(v-z\pi\alpha\rho)\geq \gamma>0$ and 
		the inequality holds.
		
		The concavity of the sine function on $(0,\frac{\pi}{2})$, implies 
		$\sin\big(u\tfrac{\pi}{2}\big)\geq u$ 
		for any $u\in(0,1)$. 
		Furthermore, $\sin(z\pi\alpha\rho)$ is symmetric on 
		$z\in[0,\frac{1}{\alpha\rho}]\supset[0,1]$ with respect to the point $ z=\frac 1{2\alpha\rho} $. Hence, using these properties, we have 
		\begin{align*}
			\E\big[e^{-xG}\big|G>0\big]
			&=\int_0^1 e^{-xg(z\pi\rho-\pi(\rho-1/2))}dz
			\le\int_0^1 e^{-\gamma x\sin(z\pi\alpha\rho)}dz
			\le\int_0^{\frac{1}{\alpha\rho}} e^{-\gamma x\sin(z\pi\alpha\rho)}dz\\
			&= 2\int_0^{\frac{1}{2\alpha\rho}} e^{-\gamma x\sin(z\pi\alpha\rho)}dz
			\le 2\int_0^{\frac{1}{2\alpha\rho}} e^{-2\gamma\alpha\rho xz}dz
			\le 2\int_0^\infty e^{-2\gamma\alpha\rho xz}dz
			=\frac{1}{\gamma\alpha\rho x}.
		\end{align*}
		The conclusion of part (b) then follows from the fact that
		$\E[e^{-xG}|G>0]\le 1$. For the second statement, it is enough to note that $ \mathbb{P}([G]^+=0)=1-\rho. $ 
		
		(c) The proof follows from elementary calculations. 
	\end{proof}
		\section{The Cauchy case $\alpha=1$}
	\label{app:cauchy}
	In this section, we will briefly remark the changes needed in all the arguments for the case $\alpha=1$ in our proofs, we proceed in the order that the arguments are presented in the main text.
	
	The Chambers-Mallows-Stuck method is not required in this case because 
	when $\alpha=1$, the stable 
	random variables $(S_k)_{k\geq1}$ have the explicit density 
	\[
	p(x)=\frac{\cos(\omega)/\pi}{\cos^2(\omega) + (x-\sin(\omega))^2},
	\qquad x\in\R.
	\]
	
	For the approximation $X_{\pm,n}$ we use $X_{\pm,n}:=\sum_{k=1}^n
	\ell_k^{1/\alpha}[S_k]^\pm+a_n \eta_\pm, $ where 
	$\eta_{\pm}$ are respectively  distributed as $\pm S_1$ 
	conditioned on the events $\{S_1>0\}$ and $\{S_1<0\}$. Note that this already hints at the fact that we will not use exponential random variables as ``length'' in this case. Instead we will use the full Cauchy random variables to  do the analysis.

	The derivative operator is defined as 
	\begin{align*}
		\D^\pm_m:=	\eta_\pm\partial_{\eta_\pm}
		\pm\sum_{k=1}^m [S_k]^\pm \partial_{S_k}.
	\end{align*}
	This operator satisfies
	\begin{align*}
		[S_k]^\pm\partial_{S_k}[X_{\pm,n}]
		&=\pm\ell_k[S_k]^\pm\1_{\{k\le n\}},\quad
 k\in\N,\enskip\text{if}\enskip\alpha=1,\\
		\D^\pm_m\big[\big(X^p_{\pm,n},f(X_{\mp,n})\big)\big]
		&=\big(pX_{\pm,n}^{p},0\big),\quad
		 m\ge n\ge 1.
	\end{align*}
	The space of smooth random variables is 
	\begin{align*}
		\s_m(\Omega)
		:=\big\{\Phi\in L^0(\Omega): 
		\enskip\exists \phi(\cdot,\vartheta)\in 
		\s_\infty((0,\infty)^{2m+2}, S), \enskip
		\Phi=\phi(\mathcal{S}_m,\mathcal{U}_m,
		\eta_+,\eta_-,\vartheta)\big\},
	\end{align*}
	where $\mathcal{S}_m:=(S_1,\ldots,S_m)$.
	The Ibpf in finite dimension gives
	\begin{align*}
		H^\pm_{n,m}(\Phi)
		:=\frac{1}{X_{\pm,n}}\Big(\Big(\frac{2\eta_\pm (\pm\eta_\pm-\sin(\omega))}
		{\cos^2(\omega)+(\eta_\pm\mp\sin(\omega))^2}
		\pm\sum_{k=1}^m
		\frac{2[S_k]^\pm(S_k-\sin(\omega))}
		{\cos^2(\omega)+(S_k-\sin(\omega))^2}\Big)\Phi
		-\D^\pm_m[\Phi]\Big).
	\end{align*}
	Note that the formula is essentially different from the one in \eqref{eq:IBP_n} as this formula is based on Cauchy random variables. 
	In the proof of Proposition \ref{prop:IBP_n}, one uses
	\begin{align*}
		\widetilde\partial_\vartheta[\omega]
		=\omega\frac{2(\vartheta-\sin(\omega))}{\cos^2(\omega)
			+(\vartheta-\sin(\omega))^2}
		-\partial_\vartheta[\omega].
	\end{align*}
	In this case recall that if $\eta$ is a Cauchy random variable with parameter $\rho$. Also one has in this case
	\begin{align*}
		\E[\Lambda_1 \eta\partial_{\eta}[\Lambda_2]]=\E\Big[\Lambda_1\Lambda_2[\eta]^\pm
		\frac{2(\eta-\sin(\omega))}{\cos^2(\omega)+(\eta-\sin(\omega))^2}
		-\Lambda_2\partial_{\eta}[\Lambda_1[\eta]^\pm]\Big]
		=\E[\Lambda_2\widetilde{\partial}_{\eta}[\Lambda_1[\eta]^\pm]].
	\end{align*}
	From here the rest of the proof of Proposition \ref{prop:IBP_n} follows similarly. The statement in Theorem \ref{thm:series} and its proof are independent of $\alpha$.
	
	The statement of the main Theorem and its proof remain unchanged if $\alpha=1$.
	
	Starting in Lemma \ref{lem:simplify_HPhi} and for the rest of the proof in this case, we let $ Z_m=1 $. The proof of this lemma in this case is as follows:
	
	Recall that  
	$\D_m^\pm[X_{\pm,n}^{p}]=pX_{\pm,n}^{p}$ and 
	$\D_m^\mp[X_{\pm,n}^p]=0$ for any $p\in\R$. Define the bounded functions 
	$q_\pm:x\mapsto 2[x]^+(\pm x-\mu)/(\gamma^2+(x\mp\mu)^2)$. 
	Recursively define the bounded functions 
	$q^{(k+1)}_\pm(x) := x \partial_x q^{(k)}_\pm(x)$ and the operators 
	$\D_m^{\pm,k+1} :=\D_m^{\pm,k}\D_m^{\pm,1}$ for $k\ge 1$, 
	where $q^{(1)}_\pm=q_\pm$ and $\D_m^{\pm,1}=\D_m^{\pm}$. 
	Let $Z^{\pm}_{k,m} = q^{(k)}_\pm(\eta_\pm) 
	+ \sum_{i=1}^m q^{(k)}_\pm([S_i]^\pm)$, we deduce that an iteration 
	of~\eqref{eq:IBP_n} yields $X_{+,n}^{-k_+}X_{-,n}^{-k_-}$ multiplied by 
	\[
	p_+(Z^+_{1,m},\ldots,Z^+_{k_+,m},\D_m^{+,1}[\Phi],\ldots,\D_m^{+,k_+}[\Phi])
	p_-(Z^-_{1,m},\ldots,Z^-_{k_-,m},\D_m^{-,1}[\Phi],\ldots,\D_m^{-,k_-}[\Phi]),
	\] 
	where $p_\pm$ are multivariate polynomials of degree $k_\pm$ 
	whose coefficients are linear in $\Phi$ and do not depend on $n$ or $m$. 
	The arguments of $p_\pm$ are uniformly bounded by $Km$ for some $K>0$ 
	independent of $n$ and $m$ (recall $\|q_\pm^{(n)}\|_\infty<\infty$ and 
	$\|\partial_+^{j_+}\partial_-^{j_-}\phi\|_\infty<\infty$ for $j_\pm\le k_\pm$). 
	so the claim follows easily.
	
{Not surprisingly the proof of the technical Proposition \ref{prop:Theta_bound} and Lemmas in Section \ref{subsec:indicators} do not depend on the fact that $ \alpha=1 $ or not. In fact, we only used algebraic properties in order to obtain these results supposing the correct moment estimates.} The required moment estimates are obtained in the next subsection.
	
	\subsection{Moment bounds for the Cauchy case $\alpha=1$}
	\label{subsec:Cauchy}
	
	For the proofs of Lemmas~\ref{lem:bound_X+X-Z},~\ref{lem:inv-mom} 
	and~\ref{lem:inv_bound_X+X-Z}, note that in comparison with 
	the Chambers-Mallows-Stuck decomposition method we are not using any 
	decomposition of the Cauchy random variables. This means that the 
	proofs will be reduced to bounding the conditional expectations with 
	respect to $(\ell_i)_{i\in\{1,\ldots,n\}}$, and then taking expectations there. 
	
	For the proof of Lemma~\ref{lem:inv-mom} (and hence of 
	Lemma~\ref{lem:inv_bound_X+X-Z}), note that in order to compute inverse 
	moments, we use a change of variable trick to handle negative moments 
	via Laplace transforms. In the present case, this implies the computation of $\E[e^{-x[S_1]^\pm}]$. This is done 
	directly using the fact that its density is known. The rest of the calculations 
	are very similar. Indeed, the version of Lemma~\ref{lem:simple-estim} for the case $\alpha=1$ contains the only 
	noticeable change. That is,
	
	\begin{lem}
		\label{lem:13}
		Suppose $\alpha=1$, then for any $s\in(0,1)$ and $x>0$ we have
		\[
		\E[\eta_+^s e^{-x\eta_+}]\le 
		\frac{\cos(\omega)}{\pi\rho}\min\bigg\{
		\int_0^\infty \frac{y^s}{\cos^2(\omega) + (y-\sin(\omega))^2}dy,
		\frac{\Gamma(s+1)}{\cos^2(\omega)x^{s+1}}\bigg\}.
		\]
		In particular, $\E[e^{-x\eta_+}]\le \min\{1,(\pi\cos(\omega)\rho x)^{-1}\}$.
	\end{lem}
	
	\begin{proof}
		Observe that for any $s\in (0,1)$, it holds that 
		\begin{align*}
			\E[\eta_+^se^{-x\eta_+}]=&\frac{\cos(\omega)}{\pi\rho}\int_0^\infty \frac{y^s}{\cos^2(\omega)+(y-\sin(\omega))^2}e^{-xy}dy\\
			\leq&\frac{\cos(\omega)}{\pi\rho}\min\left\{\int_0^\infty \frac{y^s}{\cos^2(\omega)+(y-\sin(\omega))^2}dy,
			\frac{\Gamma(s+1)}{\cos^2(\omega) x^{s+1}}\right\}.\qedhere
		\end{align*}
	\end{proof}
	Given the above change, the statements in Lemma  \ref{lem:inv-mom} are valid with the following changes: For $ u\in [0,1) $ and all other parameters as in in Lemma  \ref{lem:inv-mom} 
	\begin{align*}
		\text{(a)}\quad	\E\bigg[\frac{\ell_{n+1}^r \eta_\pm^u}{X_{+,n}^pX_{-,n}^q}\bigg]
		\leq & CT^{r-\frac{p+q}{\alpha}}(1+r)^{-n}\\
		\text{(b)}\quad	
		\E\bigg[\frac{X_{-,n}^q\ell_{n+1}^r \eta_+^u}{X_{+,n}^p}\bigg]
		\leq&CT^{r+\frac{q-p}{\alpha}}(1+r)^{-n}.
	\end{align*}The Lemmas \ref{lem:bound_X+X-Z} and \ref{lem:inv_bound_X+X-Z} remain unchanged.

	\section*{Acknowledgement}
	JGC and AM are supported by EPSRC grant EP/V009478/1 and The Alan Turing Institute under the EPSRC grant EP/N510129/1; 
    AM was supported by the Turing Fellowship funded by the Programme on Data-Centric Engineering of Lloyd's Register Foundation; 
	AK-H was supported by JSPS KAKENHI Grant Number 20K03666.


\begin{thebibliography}{GCMKH20}
		\bibitem[BC16]{bally}
		Vlad Bally and Lucia Caramellino, \emph{Stochastic integration by parts},
		Advanced Courses in Mathematics - CRM Barcelona, Springer International
		Publishing, 2016.
		
		\bibitem[BR86]{bhat}
		Rabi N. Bhattacharya and R. Ranga Rao,
		Normal Approximation and Asymptotic Expansions
		Classics in Applied Mathematics, 64, SIAM, 1986.
		
		\bibitem[BD09]{Bouleau}
		Nicolas Bouleau and Laurent Denis, \emph{Energy image density property and the
			lent particle method for poisson measures}, Journal of Functional Analysis
		\textbf{257} (2009), no.~4, 1144 -- 1174.
		
		\bibitem[BDP11]{bernyk2011}
		Violetta Bernyk, Robert~C. Dalang, and Goran Peskir, \emph{Predicting the
			ultimate supremum of a stable l\'{e}y process with no negative jumps}, Ann.
		Probab. \textbf{39} (2011), no.~6, 2385–2423.
		
		\bibitem[BGJ87]{bich}
		Klaus Bichteler, Jean-Bernard Gravereaux, and Jean Jacod, \emph{Malliavin
			calculus for processes with jumps}, Stochastic Monographs : Theory and
		Applications of Stochastic Processes, Vol 2, Gordon and Breach Science
		Publishers, 1987.
		
		\bibitem[Bin73]{MR0415780}
		Nicholas~H. Bingham, \emph{Maxima of sums of random variables and suprema of
			stable processes}, Z. Wahrscheinlichkeitstheorie und Verw. Gebiete
		\textbf{26} (1973), no.~4, 273--296. \MR{0415780 (54 \#3859)}
		
		\bibitem[Cha13]{MR3098676}
		Lo\"ic Chaumont, \emph{On the law of the supremum of {L}\'evy processes}, Ann.
		Probab. \textbf{41} (2013), no.~3A, 1191--1217. \MR{3098676}
		
		\bibitem[CM16]{MR3531705}
		Lo\"ic Chaumont and Jacek Ma{\l}ecki, \emph{On the asymptotic behavior of the
			density of the supremum of {L}\'evy processes}, Ann. Inst. Henri Poincar\'e
		Probab. Stat. \textbf{52} (2016), no.~3, 1178--1195. \MR{3531705}
		
		\bibitem[CM21]{MR4216728}
		Lo\"ic Chaumont and Jacek Ma{\l}ecki, \emph{Density behaviour related to {L}\'{e}vy processes}, Trans. Amer. Math. Soc. \textbf{374} (2021), no.~3, 1919--1945. \MR{4216728}
		
		
		\bibitem[CZ16]{MR3500272}
		Zhen-Qing Chen and Xicheng Zhang, \emph{Heat kernels and analyticity of
			non-symmetric jump diffusion semigroups}, Probab. Theory Related Fields
		\textbf{165} (2016), no.~1-2, 267--312. \MR{3500272}
		
		\bibitem[Dar56]{MR80393}
		Donald~A. Darling, \emph{The maximum of sums of stable random variables},
		Trans. Amer. Math. Soc. \textbf{83} (1956), 164--169. \MR{80393}
		
		\bibitem[DM15]{MR3379923}
		Krzysztof D\c{e}bicki and Michel Mandjes, \emph{Queues and {L}\'{e}vy
			fluctuation theory}, Universitext, Springer, Cham, 2015. \MR{3379923}
		
		\bibitem[Don08]{MR2402160}
		Ronald~A. Doney, \emph{A note on the supremum of a stable process}, Stochastics
		\textbf{80} (2008), no.~2-3, IMS Lecture Notes---Monograph Series, 151--155.
		\MR{2402160}
		
		\bibitem[DS10]{MR2599201}
		Ronald~A. Doney and Mladen~S. Savov, \emph{The asymptotic behavior of densities
			related to the supremum of a stable process}, Ann. Probab. \textbf{38}
		(2010), no.~1, 316--326. \MR{2599201}
		
		\bibitem[FP10]{fournier2010}
		Nicolas Fournier and Jacques Printems, \emph{Absolute continuity for some
			one-dimensional processes}, Bernoulli \textbf{16} (2010), no.~2, 343--360.
		
		\bibitem[GCMUB19]{MR4032169}
		Jorge~I. Gonz\'{a}lez~C\'{a}zares, Aleksandar Mijatovi\'{c}, and Ger\'{o}nimo
		Uribe~Bravo, \emph{Exact simulation of the extrema of stable processes}, Adv.
		in Appl. Probab. \textbf{51} (2019), no.~4, 967--993. \MR{4032169}
		
		\bibitem[GCMUB22]{LevySupSim}
		Jorge~I. Gonz{\'a}lez~C{\'a}zares, Aleksandar {Mijatovi{\'c}}, and Ger{\'o}nimo
		Uribe~Bravo, \emph{Geometrically convergent simulation of the extrema of
			{L}\'evy processes}, Math. Opr. Res. \textbf{47} (2022), no.~2, 1141-1168.
			
		
		\bibitem[Gil08]{MR2436856}
		Michael~B. Giles, \emph{Multilevel {M}onte {C}arlo path simulation}, Oper. Res.
		\textbf{56} (2008), no.~3, 607--617. \MR{2436856}
		
		\bibitem[Hey69]{MR251766}
		Christopher~C. Heyde, \emph{On the maximum of sums of random variables and the
			supremum functional for stable processes}, J. Appl. Probability \textbf{6}
		(1969), 419--429. \MR{251766}
		
		\bibitem[Kul19]{Kulik}
		Alexei Kulik, \emph{On weak uniqueness and distributional properties of a
			solution to an {SDE} with $\alpha$-stable noise}, Stoch. Proc. Appl.
		\textbf{129} (2019), 473–506.
		
		\bibitem[Kun19]{Kunita}
		Hiroshi Kunita, \emph{Stochastic flows and jump-diffusions}, Springer
		Singapore, 2019.
		
		\bibitem[Kuz11]{MR2789582}
		Alexey Kuznetsov, \emph{On extrema of stable processes}, Ann. Probab.
		\textbf{39} (2011), no.~3, 1027--1060. \MR{2789582}
		
		\bibitem[Kuz13]{MR3005012}
		\bysame, \emph{On the density of the supremum of a stable process}, Stochastic
		Process. Appl. \textbf{123} (2013), no.~3, 986--1003. \MR{3005012}
		
		\bibitem[NN18]{nualart2}
		David Nualart and Eulalia Nualart, \emph{Introduction to {M}alliavin
			{C}alculus}, Institute of Mathematical Statistics Textbooks, p.~158–181,
		Cambridge University Press, 2018.
		
		\bibitem[Nua06]{nual:06}
		David Nualart, \emph{The {M}alliavin calculus and related topics. probability
			and its applications}, Springer-Verlag, New York, 2006.
		
		\bibitem[Pic96]{Picard1}
		Jean Picard, \emph{On the existence of smooth densities for jump processes},
		Probab. Th. Rel. Fields \textbf{105} (1996), 481--511.
		
		\bibitem[Pic10]{Picard2}
		\bysame, \emph{Erratum to: On the existence of smooth densities for jump
			processes}, Probab. Th. Rel. Fields \textbf{147} (2010), 711–713.
		
		\bibitem[PS18]{MR3835481}
		Pierre Patie and Mladen Savov, \emph{Bernstein-gamma functions and exponential
			functionals of {L}\'{e}vy processes}, Electron. J. Probab. \textbf{23}
		(2018), Paper No. 75, 101. \MR{3835481}
		
		\bibitem[PUB12]{MR2978134}
		Jim Pitman and Ger\'onimo Uribe~Bravo, \emph{The convex minorant of a {L}\'evy
			process}, Ann. Probab. \textbf{40} (2012), no.~4, 1636--1674. \MR{2978134}
		
		
		
		\bibitem[Sat13]{Sato}
		Ken-iti Sato, \emph{L\'evy processes and infinitely divisible distributions},
		Cambridge Studies in Advanced Mathematics, vol.~68, Cambridge University
		Press, Cambridge, 2013, Translated from the 1990 Japanese original, Revised
		edition of the 1999 English translation. \MR{3185174}
		
		\bibitem[Sav20]{eSavov}
		Mladen Savov, private communication, July 2020.
		
		\bibitem[Wer96]{Weron}
		Rafa{\l} Weron, \emph{On the {C}hambers-{M}allows-{S}tuck method for simulating
			skewed stable random variables}, Statistics \& Probability Letters
		\textbf{28} (1996), no.~2, 165 -- 171.
		
		\bibitem[Zol86]{MR854867}
		V.~M. Zolotarev, \emph{One-dimensional stable distributions}, Translations of
		Mathematical Monographs, vol.~65, American Mathematical Society, Providence,
		RI, 1986, Translated from the Russian by H. H. McFaden, Translation edited by
		Ben Silver. \MR{854867}
		
		\bibitem[GCKHM22]{Presentation_AM}
		Gonz\'alez C\'azares, Jorge I. and Kohatsu Higa, Arturo and Mijatovi\'c, Aleksandar,
   \emph{Presentation on ``Joint density of the stable process and its supremum: regularity and upper bounds''} (2022), YouTube video
   {\url{https://youtu.be/x0n3Up9CxCA}}.
		
	\end{thebibliography}
\end{document}